\documentclass{amsart}
\usepackage[utf8]{inputenc}
\usepackage{amsmath}
 \usepackage{stmaryrd}
\usepackage{amssymb,latexsym}
\usepackage{amsfonts,mathrsfs}
\usepackage[margin=1.8in]{geometry}
\usepackage[all,cmtip]{xy}
\usepackage[colorlinks= true,pdfborder={0 0 0}]{hyperref}
\usepackage{graphicx}
 \usepackage{esint}
 \usepackage {verbatim}
\usepackage{color}
\def\pd#1{\dfrac{\partial}{\partial #1}}
\def\f#1#2{\frac{#1}{#2}}

\def\pa{\partial}

\def\n{\nabla}

\def\({\left (}
\def\){\right )}
\def\<{\langle}
\def\>{\rangle}
\def\ma{|\mathring{A}|^2}
\def\ch{\mathring{h}_{ij}}

\newcommand{\bel}[1]{\begin{equation}\label{#1}}

  \newcommand{\beq}{\begin{equation}}

\newcommand{\ba}{\begin{eqnarray}}
\newcommand{\ea}{\end{eqnarray}}
\newcommand{\rf}[1]{(\ref{#1})}
\newcommand{\qe}{\end{equation}}
\newcommand{\eeq}{\end{equation}}

\newtheorem{thm}{Theorem}[section]

\newtheorem{lem}[thm]{Lemma}
\newtheorem{prop}[thm]{Proposition}
\newtheorem{defn}[thm]{Definition}

\newtheorem{claim}{Claim}[section]
\newtheorem*{acknowledge}{Acknowledgement}

\newcommand{\norm}[1]{\left\Vert#1\right\Vert}
\newcommand{\abs}[1]{\left\vert#1\right\vert}
\newcommand{\set}[1]{\left\{#1\right\}}

\newcommand{\eps}{\varepsilon}

\newcommand{\red}[1]{ {\color{red} #1} }
\newcommand{\blue}[1]{ {\color{blue} #1} }
\newcommand{\tr}{\mbox{\rm tr\,}}

\title[harmonic mean curvature surface]{Foliation of constant harmonic mean curvature surfaces in asymptotically Schwarzschild spaces}
\author{Yaoting Gui, Yuqiao Li, Jun Sun}
\address{Yaoting Gui, Beijing International Mathematical Research Center, Peking University, Beijing, 100087, P.R.China}
\address{Yuqiao Li, Department of Mathematics, Hefei University of Technology, Hefei, 230009, P.R.China}
\address{Jun Sun, School of Mathematics and Statistics, Wuhan University, Wuhan, 430072, P. R. of China}
\email{ytgui@bicmr.pku.edu.cn, lyq112@mail.ustc.edu.cn, sunjun@whu.edu.cn}
\thanks{The second author is supported by Initial Scientific Research Fund of Young Teachers in Hefei university of Technology of No.JZ2024HGQA0122; The third author is supported by NSFC No. 12071352, 12271039}

\begin{document}

\begin{abstract}
    This paper investigates the volume-preserving harmonic mean curvature flow in asymptotically Schwarzschild spaces. We demonstrate the long-time existence and exponential convergence of this flow with a coordinate sphere of large radius serving as the initial surface in the asymptotically flat end, which eventually converges to a constant harmonic mean curvature surface. We also establish that these surfaces form a foliation of the space outside a large ball. Finally, we utilize this foliation to define the center of mass, proving that it agrees with the center of mass defined by the ADM formulation of the initial data set.
\end{abstract}

\maketitle

\section{Introduction}
	\allowdisplaybreaks

 A spacelike timeslice has the structure of a complete Riemannian three-manifold
 with an asymptotically flat end in the description of isolated gravitating systems in general relativity. Thus, in this paper, we are concerned with the so-called asymptotically Schwarzschild space. Precisely, a complete Riemannian manifold $(N, \bar{g})$ is said to be an asymptotically Schwarzschild space, if $(N, \bar{g})$ is a 3-manifold where $N$ is diffeomorphic to $\mathbb{R}^n\backslash B_1(0)$ and the metric $\bar{g}$ is asymptotically flat, namely,
\begin{equation}\label{metr}
\bar{g}_{\alpha\beta}=\left(1+\frac{m}{2r}\right)^4\delta_{\alpha\beta}+P_{\alpha\beta}, 
\end{equation} 
where $r$ is the Euclidean distance, $m>0$ is a constant representing the ADM mass and $P_{\alpha\beta}$ satisfies
\[ |\partial^lP_{\alpha\beta}|\leq C_{l+1}r^{-l-2}, \quad 0\leq l\leq5, \]
with $\partial$ denoting partial derivatives with respect to the Euclidean metric $\delta$. For notation convenience, we in the sequel set $C_0=\max(1, m, C_1, C_2, C_3, C_4, C_5, C_6)$.

The existence of foliation by special surfaces in an asymptotically Schwarzschild space is a significant problem in general relativity. 
The foliation offers an intrinsic geometric structure near infinity and provides a definition of the center of mass associated with an isolated physical system.
It defines a natural coordinate system near infinity which is extensively used in spacetime, see \cite{CK}.

In their seminal paper, Huisken-Yau (\cite{Huisken1996DefinitionOC}) constructed a foliation of stable constant mean curvature (CMC) surfaces in an asymptotically Schwarzschild space with $m>0$. 
These surfaces arise from the volume preserving mean curvature flow introduced by Huisken (\cite{Huisken1987TheVP}).  They also proved that the foliation is unique outside a large compact set depending on the mean curvature and defined the geometric center of mass based on their foliation. The uniqueness was further extended by Qing-Tian (\cite{QingTian}) saying that the foliation is actually unique outside of a fixed compact set, which we refer to as the global uniqueness.
Corvino-Wu proved that the geometric center of mass of the foliation is the ADM center of mass if the metric is conformally flat at infinity in \cite{Corvino}. 
The conformal flatness condition later is  removed by Huang in \cite{Huang09}.

Ye provided a different proof of the existence of CMC surfaces in \cite{Ye}.
Metzger \cite{Metzger} later generalized Huisken-Yau's result to asymptotically flat manifolds with a slightly weaker asymptotic condition, namely, the metric $\bar{g}$ in \eqref{metr} satisfies $ |\partial^lP_{\alpha\beta}|\leq C_{l+1}r^{-l-1-\epsilon}, \quad 0\leq l\leq2$ for $\epsilon\geq0$.
However, the metric there should be assumed to be rotationally symmetric while the ADM center of mass can be defined for asymptotically flat manifolds satisfying the Regge-Teitelboim (RT) condition, which is interpreted as an asymptotically symmetric condition.
Huang relaxed the symmetric assumption and proved that the CMC foliation exists in the exterior region of an asymptotically flat manifold satisfying the RT condition in \cite{Huang}.
Nerz extended the existence and uniqueness of CMC foliation in 3-dimensional asymptotically flat manifolds without asymptotic symmetry in \cite{Nerz}.
Furthermore, the corresponding result was established by Eichmair-Merzger for dimension greater than three in \cite{Eich-Met}. Finally, we also notice that the global uniqueness of the CMC foliation is characterized by Chodosh-Eichmair in \cite{Cho-Eich}.

The CMC foliation is not the only foliation utilized in mathematical general relativity.
For instance, Lamm-Metzger-Schulze establish the existence of a foliation of the asymptotically Schwarzschild space with positive mass by surfaces which are critical points of the Willmore functional subject to an area constraint in \cite{Lamm}. 
The geometric center of mass of the foliation by the area constrained Willmore surfaces is also shown to agree with the ADM center of mass if the scalar curvature is asymptotically vanishing \cite{Willmore}. 

In this paper, we will establish a novel foliation by surfaces of constant harmonic mean curvature in an asymptotically Schwarzschild space with $m>0$. These surfaces arise from the harmonic mean curvature flow (HMCF), which has been investigated in the Euclidean space \cite{MR2287150}\cite{MR2564404}. For non-Euclidean space, Xu explored the unnormalized harmonic mean curvature flow in his thesis (\cite{Xuguoyi}). More general mixed volume preserving curvature flows are investigated in Euclidean space in \cite{McCoy2005MixedVP} and in hyperbolic space in \cite{Wang-xia}.

Different from their attempt, we will examine the volume preserving harmonic mean curvature flow in asymptotically Schwarzschild spaces as a generalization of Huisken-Yau's result \cite{Huisken1996DefinitionOC}. To be precise, we first introduce some fundamental notion. 
Let $S_\sigma(0)$ be the coordinate sphere of radius $\sigma$ centered at the origin in $(N,\bar{g})$ given by a map
$$
\phi_0^\sigma:S^2\to N,\quad\phi_0^\sigma(S^2)=S_\sigma(0).
$$
We aim to find a one-parameter family of maps $\phi_t^\sigma=\phi^\sigma(\cdot,t)$ solving the initial value problem
\bel{flow}
\begin{cases}
\frac{d}{dt}\phi^{\sigma}(p,t)&=(f-F)\nu(p,t),\quad t\geq0,\:p\in S^{2},\\
\phi^{\sigma}(0)&=\phi_{0}^{\sigma},
\end{cases}
\qe
where $F(p, t)=\frac{H^2-|A|^2}{2H}$ and $\nu(p, t)$ represent the harmonic mean curvature and the outward unit normal vector of the surface $\Sigma_t=\phi_t^\sigma(S^2)$ respectively. Here we denote $|\Sigma_t|$ and $f(t)=\frac{\int_{\Sigma_t}Fd\mu_t}{|\Sigma_t|}$ the area and the average of the harmonic mean curvature.

The flow described by equation \eqref{flow} preserves the volume of the region enclosed by $\Sigma_t$ and $S_{\frac{m}{2}}(0)$. And we designate this flow \eqref{flow} as the {\em volume preserving harmonic mean curvature flow}. 
It should be noted that, owing to the additional curvature terms, there is no monotonic quantity associated with this flow.

The first result of this paper establishes the long time existence and exponential convergence of the flow \eqref{flow}, which converges to a constant harmonic mean curvature surface.

\begin{thm}\label{main}
    If $(N, \bar{g})$ is the asymptotically Schwarzschild manifold with metric \eqref{metr}, then there is $\sigma_0$ depending only on $C_0$ such that for all $\sigma\geq\sigma_0$, the volume preserving harmonic mean curvature flow \eqref{flow} has a smooth solution for all times $t\geq0$.
    As $t\rightarrow\infty$, the surfaces $\Sigma_t$ converge exponentially fast to a surface $\Sigma_{\sigma}$ of constant harmonic mean curvature $f_{\sigma}$.
\end{thm}

Once  establishing the existence of the constant harmonic mean curvature surfaces, we can further prove that they form a foliation of the asymptotically Schwarzschild manifold outside a large ball. 
For $\sigma\geq1$ and $B_1, B_2, B_3$ nonnegative numbers, we now define a set
 $\mathcal{B}_\sigma(B_1,B_2,B_3)$ of nearly round surfaces in $(N,\bar{g})$ by setting 
 \[
 \mathcal{B}_\sigma:=\set{M\subset N \big{|} \sigma-B_1\leq r\leq\sigma+B_1, |\mathring{A}|\leq B_2\sigma^{-3}, |\n\mathring{A}|\leq B_3\sigma^{-4}},
 \]
where $\mathring{A}$ is the traceless part of the second fundamental form.

\begin{thm}\label{main2}
    There is $\sigma_0$ depending only on $C_0$ and $c=c(n)$ such that for all $\sigma\geq\sigma_0$, the constant harmonic mean curvature surfaces $\Sigma_{\sigma}$ constructed in Theorem \ref{main} form a proper foliation of $N\backslash B_{\sigma_0}(0)$. 
    \iffalse Moreover, given constants $B_1, B_2, B_3$, we can choose $\sigma\geq \sigma_0\geq c(C_0^2+B_1^2+B_2+B_3)$ such that $\Sigma_{\sigma}$ is the only surface with constant harmonic mean curvature $f_{\sigma}$ contained in $\mathcal{B}_{\sigma}(B_1, B_2, B_3)$.\fi
\end{thm}

As a consequence, we are able to define the center of mass $C_{HM}$.

\begin{defn}
    Let $\Sigma_{\sigma}$ be the family of surfaces constructed in Theorem \ref{main2} and $\phi^{\sigma}$ be the position vector.
     The center of gravity of $\Sigma_{\sigma}$ is defined by 
    \[ C_{HM}=\lim_{\sigma\rightarrow\infty}\frac{\int_{\Sigma_{\sigma}}\phi^{\sigma} d\mu_e}{\int_{\Sigma_{\sigma}}d\mu_e}, \]
    where $\mu_e$ is the volume element with respect to the Euclidean metric.
\end{defn}

Our last result in this paper shows that our definition of the center of mass agrees with the center of mass defined by the ADM formulation of the initial value problem for the Einstein equation, namely, the ADM center of mass.

\begin{thm}\label{thm-center}
    If $(N, \bar{g})$ is the asymptotically Schwarzschild manifold with metric \eqref{metr}, then $C_{HM}$ is exactly the ADM center of mass.
\end{thm}

The strategies of the proof for the main theorems follow the approach of Huisken-Yau (\cite{Huisken1996DefinitionOC}) regarding the volume preserving mean curvature flow. More specifically, we initiate the volume preserving harmonic mean curvature flow \eqref{flow} from a large coordinate sphere and provide the necessary a priori estimates of the second fundamental form together with the derivatives on the circular domain $\mathcal{B}_{\sigma}(B_1, B_2, B_3)$. However, there are several crucial innovations in our paper:

Firstly, in contrast to the volume preserving mean curvature flow, which is a quasilinear parabolic system, the volume preserving harmonic mean curvature flow is a fully nonlinear system. Consequently, this complicates the computations substantially.
Secondly, the principal part of the linearized operator of the harmonic mean curvature ${\mathcal L}$ is not in divergence form. Nevertheless, it is not avoidable to give a priori estimate for the non-divergence elliptic operator ${\mathcal L}$ when proving the exponential convergence and the formation of foliation. In this context, we treat ${\mathcal L}$ as a perturbation of the Laplacian operator and furnish the estimate for the error term. Thirdly, in order to prove that $\Sigma_{\sigma}$ forms a foliation, we need to derive that the lapse function has a sign. Since the operator ${\mathcal L}$ is not self-adjoint, many useful tools in functional analysis and PDE cannot be applied. In order to obtain the desired estimate, we consider the sum of ${\mathcal L}$ and its adjoint operator ${\mathcal L}^{*}$, and establish the eigenvalue estimates for such a self-adjoint operator. Then we apply the elliptic estimate for such an operator to obtain the desired information.

%There are extreme obstacles in  establishing the global uniqueness of the constant harmonic mean curvature foliation as presented in section 5 of \cite{Huisken1996DefinitionOC} or \cite{QingTian}.We will show the uniqueness in our forthcoming paper. The case of asymptotically flat manifolds with mild decay metric or with the RT condition as in \cite{Huang} are under consideration.Moreover, we believe that  the harmonic mean curvature foliation is also applicablein asymptotically hyperbolic manifolds as described in \cite{nevetian1}\cite{nevetian2}.
Establishing the global uniqueness of the constant harmonic mean curvature foliation, as presented in Section 5 of \cite{Huisken1996DefinitionOC} or \cite{QingTian}, poses significant challenges. We will demonstrate uniqueness in our forthcoming paper, focusing on asymptotically flat manifolds with mild decay metrics or those satisfying the RT condition, as discussed in \cite{Huang}. Furthermore, we believe the harmonic mean curvature foliation is also applicable to asymptotically hyperbolic manifolds, as outlined in \cite{nevetian1}\cite{nevetian2}.

The subsequent sections are organized as follows: Section 2 provides some preliminary materials and evolution equations of geometric quantities for the volume preserving harmonic mean curvature flow; in Section 3, we give the proof of the long time existence and exponential convergence of the flow; the construction of the foliation and definition of the center of mass is treated in Section 4. 

\begin{acknowledge}
    The authors appreciate the references provided by Prof. Ma Shiguang, Prof. Zhang Xiao, Prof. Wang Zuoqin and Prof. Li Jiayu. The authors would like to express their thanks to Prof. Tian Gang for supporting this research and the hospitality of Peking University.
\end{acknowledge}

\vspace{.2in}

\section{Preliminaries}
In this section, we present some fundamental concepts and also establish some useful lemmas.
Let $\Sigma_t$ be a family of hypersurfaces in $(N, \bar{g})$ solving \eqref{flow}. Let $\nu$ be the outward unit normal to $\Sigma_t$ and $g$ be the induced metric on $\Sigma_t$. Denote the covariant derivative on $\Sigma_t$ and $N$ by $\nabla$ and $\bar{\nabla}$, respectively. Then for a fixed time $t$, we choose an orthonormal frame $e_1, e_2, \nu$ adapted to $\Sigma_t$ such that the second fundamental form $A=\{h_{ij}\}$ is given by 
\[ h_{ij}=\langle\bar{\nabla}_{e_i}\nu, e_j\rangle. \]
Let $\lambda_1, \lambda_2$ be the two principal curvatures on $\Sigma_t$ and write
\[ H=g^{ij}h_{ij}=\lambda_1+\lambda_2, \]
\[ |A|^2=h^{ij}h_{ij}=\lambda_1^2+\lambda_2^2, \]
\[ F=\frac{H^2-|A|^2}{2H}=\frac{\lambda_1\lambda_2}{\lambda_1+\lambda_2} \]
for the mean curvature, the square norm of the second fundamental form, and the harmonic mean curvature on $\Sigma_t$ respectively. 
Furthermore, let $\mathring{A}$ be the traceless second fundamental form such that
\[ \mathring{h}_{ij}=h_{ij}-\frac{1}{2}Hg_{ij}, \]
\[ |\mathring{A}|^2=|A|^2-\frac{1}{2}H^2=\frac{1}{2}(\lambda_1-\lambda_2)^2. \]

The curvature of $N$ together with the derivatives and the second fundamental form of $\Sigma_t$ are related by
the equations of Gauss and Codazzi:
\[
\begin{aligned}
R_{ijkl}&=\bar{R}_{ijkl}+h_{il}h_{jk}-h_{ik}h_{jl},\\
-\bar{R}_{3ijk}&=\nabla_{k}h_{ij}-\nabla_{j}h_{ik},
\end{aligned}
\]
where the index 3 indicates the $\nu$ direction. 
The derivatives of the curvature are related by 
\[
\bar{\n}_l\bar{R}_{3ijk}=\n_l\bar{R}_{3ijk}+h_{jl}\bar{R}_{3i3k}+h_{kl}\bar{R}_{3ij3}-h_{ml}\bar{R}_{mijk},
\]
and the commutation of derivatives are
\[
h_{ij,kl}=h_{ij,lk}+h_{im}R_{mjlk}+h_{jm}R_{milk}.
\]

We should state another consequence of the Codazzi equation, which will be utilized in dealing with the curvature estimate, especially to handle the extra terms coming from the derivative of $F$.
\begin{lem}\label{lem1.4}
If $w_{i}=\bar{R}_{3lil}$ denotes the projection of $\bar{R}ic
(\nu,\cdot)$ onto $\Sigma_t$, we have for any $\eta>0$ the inequality
\[
|\nabla A|^{2}\geq(\frac{3}{4}-\eta)|\nabla H|^{2}-(\frac{1}{4}\eta^{-1}-1)|w|^{2}.
\]
Furthermore, we have the higher order version, for $k=1, 2, 3,\cdots$,
\[
|\n^kA|^2\geq(\f34-\eta)|\n^kH|^2-(\frac{1}{4}\eta^{-1}-1)|\n^{k-1}w|^2.
\]
\end{lem}
\begin{proof}
    The first result comes from \cite{Huisken1986ContractingCH}. To show the second statement, we first consider $k=2$ and then utilize induction argument. Similar to \cite{Huisken1986ContractingCH}, we define 
    \begin{gather*}
E_{ijkl}=\f14(H_{il}g_{jk}+H_{jl}g_{ik}+H_{kl}g_{ij})+\f12(w_{kl}g_{ij}-w_{il}g_{jk}-w_{jl}g_{ik}),\\
F_{ijkl}=h_{ij,kl}-E_{ijkl}.
     \end{gather*}
    We find that 
    \begin{gather*}
    \abs{E}^2=\f34H_{kl}^2-w_{kl}H_{kl}+w_{kl}^2,\\
   \<E_{ijkl},h_{ij,kl}\>=\abs{E}^2.
      \end{gather*}
      Thus, we conclude that $\<E_{ijkl},F_{ijkl}\>=0$. It follows that 
      \begin{align*}
      |\n^2A|^2 & =|\n^2\mathring{A}|^2+\f12\abs{\n^2H}^2\\
      & = \abs{E}^2+\abs{F}^2\geq\abs{E}^2\\
      & =\f34H_{kl}^2-w_{kl}H_{kl}+w_{kl}^2,
      \end{align*}
      which implies that 
      \[
      |\n^2A|^2\geq(\frac{3}{4}-\eta)|\n^2H|^2-(\frac{1}{4}\eta^{-1}-1)|\n\omega|^2.
      \]
\end{proof}

Next we state the Simons' identity of the harmonic mean curvature $F$ computed similarly as in \cite{Schoen1975CurvatureEF} and \cite{geometry}, and an important consequence of the Codazzi equations \cite{Huisken1987TheVP} by considering $F$ as a function of $h^i_j=g^{ik}h_{jk}$.
\begin{lem}\label{lem1.2}
The harmonic mean curvature $F$ satisfies the identity
\begin{align*}
F^{kl}h_{ij;kl}  = &\n_i\n_jF+F^{kl}h_{im}h_{mj}h_{kl}-F^{kl}h_{mk}h_{ml}h_{ij}-F^{kl}h_{kl}\bar{R}_{3i3j}+h_{ij}F^{kl}\bar{R}_{3k3l}\\
&+F^{kl}(h_{jm}\bar{R}_{mkli}+h_{mi}\bar{R}_{mklj}-2h_{mk}\bar{R}_{mjil})+F^{kl}(\bar{\n}_l\bar{R}_{3jki}+\bar{\n}_i\bar{R}_{3klj})\\
&-F^{kl,pq}\n_ih_{pq}\n_jh_{kl}
\end{align*}
Here $F^{kl}(A)B_{kl}=\pd s|_{s=0}F(A+sB)$.
\end{lem}
\begin{proof}
    Direct computation gives 
    \begin{align*}
       & \n_i\n_jF  \\
       = &F^{kl}h_{kl;ji}+F^{kl,pq}\n_ih_{pq}\n_jh_{kl}\\
     %  & = F^{kl}(h_{kl;ji}+h_{km}R_{mlij}+h_{ml}R_{mkij})+F^{kl,pq}\n_ih_{pq}\n_jh_{kl}\\
      =& F^{kl}(h_{kj;li}-\n_i\bar{R}_{3klj})+F^{kl,pq}\n_ih_{pq}\n_jh_{kl}\\
         =& F^{kl}(h_{kj;il}-\n_i\bar{R}_{3klj}+h_{km}R_{mjil}+h_{jm}R_{mkil})+F^{kl,pq}\n_ih_{pq}\n_jh_{kl}\\
         =& F^{kl}(h_{ji;kl}-\n_l\bar{R}_{3jki}-\n_i\bar{R}_{3klj}+h_{km}R_{mjil}+h_{jm}R_{mkil})+F^{kl,pq}\n_ih_{pq}\n_jh_{kl}\\
      %  & + F^{kl}(h_{km}R_{mlij}+h_{ml}R_{mkij}+h_{jm}R_{mkli}+h_{mk}R_{mjli})\\
        = & F^{kl}(h_{ij;kl}-\bar{\n}_l\bar{R}_{3jki}-\bar{\n}_i\bar{R}_{3klj}+h_{kl}\bar{R}_{3j3i}+h_{ij}\bar{R}_{3kl3})\\
                & + F^{kl}(2h_{mk}\bar{R}_{mjil}-h_{jm}\bar{R}_{mkli}-h_{mi}\bar{R}_{mklj})+\underline{F^{kl}(h_{jm}h_{ml}h_{ki}-h_{km}h_{mi}h_{jl})}\\
        & + F^{kl,pq}\n_ih_{pq}\n_jh_{kl}+F^{kl}h_{mk}h_{ml}h_{ij}-F^{kl}h_{mi}h_{mj}h_{kl}.
    \end{align*}
    The underline part is disappeared since $F^{kl}$ is diagonal.
\end{proof}

Now we calculate the evolution equations of the flow \eqref{flow}.

\begin{lem}\label{evo}
     We have the evolution equations along the flow \eqref{flow}:
     \begin{enumerate}
         \item
$
\frac{\partial}{\partial t}g_{ij}=2(f-F)h_{ij},$\\
\item 
$
\frac{\partial}{\partial t}\nu=\nabla F,$\\
\item
$
\frac{\partial}{\partial t}d\mu=H(f-F)d\mu,$\\
\item
$\frac {\partial}{\partial t}h_{ij}=\nabla_{i}\nabla_jF+(f-F)h_{il}h_{lj}+(f-F)\bar{R}_{3i3j},$\\
\item
$\frac{\partial H}{\partial t} =\Delta F-(f-F)\left(\abs{A}^2+\bar{R}ic(\nu,\nu)\right),$
\end{enumerate}
where $d\mu$ is the induced measure on $\Sigma_t$.
\end{lem}

Using Lemma \ref{lem1.2}, we can readily derive the following additional equations, employing a method analogous to that in \cite{Huisken1987TheVP}. We denote the linearized operator by $\mathcal{L}f=F^{kl}f_{kl}$. It is straightforward to compute that $F^{kl}h_{kl}=F$.
\begin{lem}\label{lem1.5}
The second fundamental form satisfies the evolution equations 
\begin{enumerate}
\item 
\begin{align*}
\frac{\partial}{\partial t}h_{ij}  = & \mathcal{L}h_{ij}-(2F-f)h_{im}h_{mj}+F^{kl}h_{mk}h_{ml}h_{ij}-F^{kl}\bar{R}_{3k3l}h_{ij}\\
& -F^{kl}(\bar{\n}_l\bar{R}_{3jki}+\bar{\n}_i\bar{R}_{3klj})+f\bar{R}_{3i3j}+F^{kl,pq}\n_ih_{pq}\n_jh_{kl}\\
& +F^{kl}(2h_{km}\bar{R}_{mjil}-h_{jm}\bar{R}_{mkli}-h_{im}\bar{R}_{mklj}), 
\end{align*}

\item 
\begin{align*}
\frac{\partial}{\partial t}H  = & \mathcal{L}H-f(|A|^{2}+\bar{R}ic(\nu,\nu))-2F^{kl}(h_m^i\bar{R}_{mkli}-h_{lm}\bar{R}_{miik})\\
& -F^{kl}(\bar{\n}_l\bar{R}_{3iki}+\bar{\n}^j\bar{R}_{3klj})+HF^{kl}(h_{mk}h_{ml}-\bar{R}_{3k3l})\\
& +F^{kl,pq}\n_ih_{pq}\n^ih_{kl},
\end{align*}

\item 
\begin{align*}
   \frac{\partial}{\partial t}\ch  =&  \mathcal{L}\ch+(f-2F)\mathring{h}_{mi}\mathring{h}_{mj}-FH\mathring{h}_{ij}+\f{H^2}4F^{kl}g_{kl}\ch+\f f2\ma g_{ij}\\
   & +\red{HF^{kl}\mathring{h}_{kl}\mathring{h}_{ij}}+F^{kl}\mathring{h}_{mk}\mathring{h}_{ml}\ch-F^{kl}\bar{R}_{3k3l}\ch+f(\bar{R}_{3i3j}+\f12\bar{R}ic(\nu,\nu)g_{ij})\\
   & - F^{kl}\left(\bar{\n}_l\bar{R}_{3jki}+\bar{\n}_i\bar{R}_{3klj}-\f{\bar{\n}_l\bar{R}_{3mkm}+\bar{\n}^m\bar{R}_{3klm}}2g_{ij}\right)\\
   & +F^{kl,pq}\left(\n_ih_{pq}\n_jh_{kl}-\f{\n_mh_{pq}\n^mh_{kl}}2g_{ij}\right)\\ 
   & +F^{kl}\left(2\mathring{h}_{km}\bar{R}_{mjil}-\mathring{h}_{jm}\bar{R}_{mkli}-\mathring{h}_{im}\bar{R}_{mklj}\right)\\
   &-F^{kl}(\mathring{h}_{lm}\bar{R}_{mppk}-\mathring{h}_{mp}\bar{R}_{mklp})g_{ij},
\end{align*}

\item
\begin{align*}
\frac{\partial}{\partial t}|A|^2  = &\mathcal{L}|A|^{2}-2F^{kl}h_{ij,k}h_{ij,l}-2f\tr(A^3)+2F^{kl}h_{mk}h_{ml}\abs{A}^2 \\
&+2fh_{ij}\bar{R}_{3i3j}-2|A|^{2}F^{kl}\bar{R}_{3k3l}-4F^{kl}h_{ij}(h_{jm}\bar{R}_{mkli}-h_{km}\bar{R}_{mjil}) \\
&-2F^{kl}h_{ij}(\bar{\nabla}_{l}\bar{R}_{3jki}+\bar{\nabla}_{i}\bar{R}_{3klj})+2F^{kl,pq}h_{ij}\n_ih_{kl}\n_jh_{pq},
\end{align*}

\item 
\begin{align*}
   \frac{\partial}{\partial t}|\mathring{A}|^2  = & \mathcal{L}(\ma)+F^{kl}(H_kH_l-2h_{ij,k}h_{ij,l})+f(H\abs{A}^2-2\tr{A^3})\\
    & +2F^{kl}h_{mk}h_{ml}|\mathring{A}|^2+F^{kl,pq}(2h_{ij}\n_ih_{kl}\n_jh_{pq}-H\n_ih_{kl}\n_ih_{pq})\\
     & +\red{HF^{kl}\mathring{h}_{kl}\ma}+2fh_{ij}\bar{R}_{3i3j}+fH\bar{R}ic(\nu,\nu)-2F^{kl}\bar{R}_{3k3l}\ma\\
    & -F^{kl}\Big(2h_{ij}(\bar{\nabla}_{l}\bar{R}_{3jki}+\bar{\nabla}_{i}\bar{R}_{3klj})-H(\n_l\bar{R}_{3jkj}+\n^j\bar{R}_{3klj})\Big)\\
    & -2F^{kl}\Big(2h_{ij}(h_{im}\bar{R}_{mklj}-h_{km}\bar{R}_{mjil})-H(h_{mi}\bar{R}_{mkli}-h_{lm}\bar{R}_{miik})\Big)\\
     =  & \mathcal{L}(\ma)-2F^{kl}\mathring{h}_{ij,k}\mathring{h}_{ij,l}-2fH\ma+2F^{kl}h_{mk}h_{ml}\ma\\
    & +2f\mathring{h}_{ij}\bar{R}_{3i3j}-2F^{kl}\bar{R}_{3k3l}\ma-2F^{kl}\mathring{h}_{ij}(\bar{\n}_l\bar{R}_{3jki}+\bar{\n}_i\bar{R}_{3klj})\\
    & -4F^{kl}\mathring{h}_{ij}(h_{jm}\bar{R}_{mkli}-h_{km}\bar{R}_{mjil})+2F^{kl,pq}\ch\n_ih_{kl}\n_jh_{pq}.
\end{align*}
\end{enumerate}
\end{lem}

\vspace{.2in}

\section{Long time existence and exponential convergence}
In this section, we show the long time existence and expoential convergence of the flow \eqref{flow}. We begin by recalling the following description of round surfaces in $N$ presented in \cite{Huisken1996DefinitionOC}.

\begin{prop}\label{2.1} (Proposition 2.1 of \cite{Huisken1996DefinitionOC})
Suppose $M$ is a hypersurface in $(N, \bar{g})$ such that $r(y)\geq \frac{1}{10}\max_M r=:r_1$ for all $y\in M$ and such that for some constants $K_1, K_2,$
\[
|\mathring{A}|\leq K_1r_1^{-3},\quad |\nabla\mathring{A}|\leq K_2r_1^{-4}.
\]
Then there is an absolute constant $c$ such that the curvature $A^e$ of $M$ with respect to the Euclidean metric satisfies
\[
|\mathring{A}^e|\leq c(K_1+C_0)r_1^{-3},\quad |\nabla^eA^e|\leq c(K_2+C_0)r_1^{-4},
\]
provided $r_1\geq c(C_0+K_1)$. Moreover, there is a number $r_0\in {\mathbb R}$ and a vector $\vec{a}\in {\mathbb R}^3$ such that
\[
|\lambda_i^e-r_0^{-1}|\leq c(K_1
+K_2+C_0)r_1^{-3},\quad i=1,2,
\]
\[
|(y-\vec{a})-r_0\nu_{e}|\leq c(K_1
+K_2+C_0)r_1^{-1},
\]
\[
|\nu_{e}-r_0^{-1}(y-\vec{a})|\leq c(K_1
+K_2+C_0)r_1^{-2}.
\]
Here $y$ and $\nu_e$ are the position vector and the unit normal of $M$ in ${\mathbb R}^3$, respectively.
\end{prop}

For $\sigma\geq1$ and $B_1, B_2, B_3$ nonnegative numbers, we define a set $\mathcal{B}_{\sigma}(B_1, B_2, B_3)$ of  round surfaces in $(N, \bar{g})$ by setting
\[ \mathcal{B}_{\sigma}:=\{ M\subset N|\sigma-B_1\leq r\leq \sigma+B_1, |\mathring{A}|\leq B_2\sigma^{-3}, |\nabla\mathring{A}|\leq B_3\sigma^{-4} \}. \]

The following proposition provides the precise information regarding the mean curvature of surfaces in $\mathcal{B}_{\sigma}(B_1, B_2, B_3)$.

\begin{prop}\label{2.2} (Proposition 2.2 of \cite{Huisken1996DefinitionOC})
Let $M$ be a surface in $\mathcal{B}_{\sigma}(B_1, B_2, B_3)$. Suppose $\sigma\geq c(B_1+B_2+C_0)$ is such that all assumptions of Proposition \ref{2.1} are satisfied
and let $r_0$ and $\vec{a}$ be as in that proposition. Then there is an absolute constant $c$ such that the principal curvatures and mean curvature of $M$ satisfies
\[
\left|\lambda_i-\frac{1}{r_0}+\frac{2m}{r_0^2}-\frac{3m\langle \vec{a}, \nu_e\rangle_e}{r_0^3}\right|\leq c(C_0^2+B_2+B_3)\sigma^{-3},\quad i=1,2,
\]
\[
\left|H-\frac{2}{r_0}+\frac{4m}{r_0^2}-\frac{6m\langle \vec{a}, \nu_e\rangle_e}{r_0^3}\right|\leq c(C_0^2+B_2+B_3)\sigma^{-3},
\]
provided $\sigma\geq c(B_1^2+B_2+C_0)$.
\end{prop}

For the proof of long time existence, we will first demonstrate $\Sigma_t$ remains round during the flow \eqref{flow}.

\begin{thm}\label{3.3}
    There are constants $\sigma_0, B_1, B_2, B_3$ depending only on $C_0$ such that for all $\sigma\geq\sigma_0$, the solution $\Sigma_t$ of \eqref{flow} remains in $\mathcal{B}_{\sigma}(B_1, B_2, B_3)$ as long as it exists.
\end{thm}

We will always assume that $\sigma, B_1, B_2, B_3$ are chosen such that $\Sigma_0=S_{\sigma}(0)$  is strictly inside $\mathcal{B}_{\sigma}(B_1, B_2, B_3)$.
In the next two propositions we control the position of $\Sigma_t$ during the evolution.
Since the flow \eqref{flow} preserves the volume of the region enclosed by $\Sigma_t$ and a fixed surface $S_1(0)$, the proof of Proposition 3.4 follows similarly as in \cite{Huisken1996DefinitionOC}.

\begin{prop}\label{3.4}(Proposition 3.4 in \cite{Huisken1996DefinitionOC})
    Suppose that $\Sigma_t$ is a smooth solution of the flow \eqref{flow} contained in $\mathcal{B}_{\sigma}(B_1, B_2, B_3)$ for all $t\in [0, T]$. Assume that $\sigma\geq c(C_0+B_1+B_2)$ is such that Proposition \ref{2.1} applies and $r_0(t)$ is defined as in that result. Then there is an absolute constant $c$ such that
    \[|r_0(t)-\sigma|\leq c(C_0+B_2+B_3) \]
    holds uniformly in $[0, T]$, provided that $\sigma\geq c(C_0+B_1+B_2)$.
\end{prop}

\begin{prop}\label{3.5}
    Suppose that the solution $\Sigma_t$ of \eqref{flow} is contained in $\mathcal{B}_{\sigma}(B_1, B_2, B_3)$ for all $t\in [0, T]$. Then there is an absolute constant $c$ such that
    \[ \max_{\Sigma_t}r\leq \sigma+c(m^{-1}+1)(C_0^2+B_2+B_3) \]
     holds uniformly in $[0, T]$, provided that $\sigma\geq c(C_0^2+B_1^2+B_2+B_3)$.
\end{prop}

\begin{proof}
    Let $D>0$ and assume $\max_{\Sigma_t}r<\sigma+D$ is violated for the first time at $(y_0, t_0),\  t_0>0$.
    At that point $r(y_0, t_0)=\max_{\Sigma_t}r=\sigma+D$, $\langle y_0, \nu_e\rangle_e=r$ and 
   $(f-F)\langle \nu, \nu_e\rangle_e\geq0$.
    From Proposition \ref{2.1}, it follows that $\langle \nu, \nu_e\rangle_e\geq\frac{1}{2}$ and $\langle \Vec{a}, \nu_e\rangle_e\geq\frac{1}{2}|\vec{a}|$ at $(y_0, t_0)$ provided $\sigma\geq c(C_0+B_2+B_3)$. 
    So we have $F\leq f$ at $(y_0, t_0)$, but we get from Proposition \ref{2.2} that
    \begin{equation*}
        F=\frac{H^2-|A|^2}{2H}=\frac{H}{4}-\frac{|\mathring{A}|^2}{2H}=\frac{1}{2r_0}-\frac{m}{r_0^2}+\frac{3m\langle \vec{a}, \nu_e\rangle_e}{2r_0^3}+O((C_0^2+B_2+B_3)\sigma^{-3}),
    \end{equation*} 
    Thus, we get at $(y_0,t_0)$,
    \begin{equation}\label{3.2}
    \begin{split}
        F-f =& \frac{3m\langle \vec{a}, \nu_e\rangle_e}{2r_0^3}-\frac{3m}{2r_0^3}\oint_{\Sigma_t}\langle \vec{a}, \nu_e\rangle_ed\mu_t+O((C_0^2+B_2+B_3)\sigma^{-3})\\
        \geq& \frac{3m|\vec{a}|}{4r_0^3}-c(C_0^2+B_2+B_3)\sigma^{-3}\\
        \geq& (\frac{m|\vec{a}|}{2}-c(C_0^2+B_2+B_3))\sigma^{-3},
    \end{split}
    \end{equation} 
    provided that $\sigma\geq c(C_0^2+B_1^2+B_2+B_3)$. At $(y_0, t_0)$, in view of Proposition \ref{3.4},
     we have
     \[ |\vec{a}|\geq D-c(C_0+B_2+B_3). \]
     Combining with \eqref{3.2}, we finally get
     \[ 0\geq F-f\geq \left(\frac{m}{2}D-(\frac{m}{2}+1)c(C_0^2+B_2+B_3)\right)\sigma^{-3}, \]
     which yields a contradiction if $D>c(m^{-1}+1)(C_0^2+B_2+B_3)$.
\end{proof}

\begin{lem}\label{lem1.7}
    Let $\Sigma_t$ be a smooth solution of \eqref{flow} contained in $\mathcal{B}_\sigma(B_1,B_2,B_3)$,
for $t\in[0,T]$. Then there is an absolute constant c such that for $\sigma\geq c (C_{0}^{2}+B_1^2+B_2)$
\begin{enumerate}
\item
\[\left|\lambda_i-\frac1\sigma\right|+|F-f| \leq c (C_0^2+B_2+B_3)\sigma^{-2},\]
\item
\[
\frac{2f}{H^3}(|A|^4-H tr(A^3)) \leq-\frac18|\mathring{A} |^2, 
\]
\item
\[
\frac f{H^2}|\mathring{h}_{ij}\bar{R}_{i3j3}| \leq cC_0|\mathring{A} |\sigma^{-3}, 
\]
\item
\[
\left|F^{kl}\left({h_{ij}-\f{\abs{A}^2}{H}g_{ij}}\right)(h_{jm}\bar{R}_{mkil}-h_{lm}\bar{R}_{imjk})\right| \leq cC_0|\mathring{A}\mid^2\sigma^{-3}, 
\]
\item
\[
\left|F^{kl}\left({h_{ij}-\f{\abs{A}^2}{H}g_{ij}}\right)(\bar{\nabla}_{l}\bar{R}_{3jki}+\bar{\nabla}_{i}\bar{R}_{3klj})\right| \leq cC_0|\mathring{A} |\sigma^{-5}, 
\]
\item 
\begin{align*}
\left|F^{kl,pq}\left({h_{ij}-\f{\abs{A}^2}{H}g_{ij}}\right)\n_ih_{kl}\n_jh_{pq}\right|\leq c(C_0^2+B_1^2+B_3^2)|\mathring{A}|\sigma^{-7}.
\end{align*}
\end{enumerate}
\end{lem}
\begin{proof}
    The first inequality is from Proposition \ref{2.2}, \ref{3.4}. To see the second inequality, observe that 
    \[
    |A|^4-H \tr(A^3)=-\lambda_1\lambda_2(\lambda_1-\lambda_2)^2=-2\lambda_1\lambda_2|\mathring{A} |^2,
    \]
    which implies that
    \[ \frac{2f}{H^3}(|A|^4-H tr(A^3))=(\frac{\sigma^2}{8}+c\sigma)(-2\lambda_1\lambda_2|\mathring{A} |^2)\leq -\frac{1}{8}|\mathring{A}|^2. \] 
    Now recall that the Riemann curvature tensor on a three-dimensional manifold can be expressed by the Ricci curvature:
    \bel{for1}
    \bar R_{\alpha\beta\gamma\delta}=(\bar R_{\alpha\delta}\bar g_{\beta\gamma}-\bar R_{\alpha\gamma}\bar g_{\beta\delta}-\bar R_{\beta\delta}\bar g_{\alpha\gamma}+\bar R_{\beta\gamma}\bar g_{\alpha\delta})+\frac{1}{2}\bar R(\bar g_{\alpha\gamma}\bar g_{\beta\delta}-\bar g_{\alpha\delta}\bar g_{\beta\gamma}).
    \qe
    It follows that
$$
\mathring{h}_{ij}\bar{R}_{i3j3}=\frac{1}{2}(\lambda_{1}-\lambda_{2})(\bar{R}_{11}-\bar{R}_{22}).
$$
But from Proposition \ref{2.1}, we get $|\bar{R}_{11}-\bar{R}_{22}|\leq cC_{0}\sigma^{-4}$ for $\sigma\geq c(C_0+B_1+B_2+B_3)$ as desired. 
The fourth estimate follows from the following
\begin{align*}
    F^{kl}\left({h_{ij}-\f{\abs{A}^2}{H}g_{ij}}\right)(h_{jm}\bar{R}_{mkil}-h_{lm}\bar{R}_{imjk})
    =\f{(\lambda_1-\lambda_2)^2\lambda_1\lambda_2}{H^2}\bar{R}_{1212}.
\end{align*}
Finally using \eqref{for1} again, we obtain
\begin{align*}
    F^{kl}\left({h_{ij}-\f{\abs{A}^2}{H}g_{ij}}\right)(\bar{\nabla}_{l}\bar{R}_{3jki}+\bar{\nabla}_{i}\bar{R}_{3klj})=\f{(\lambda_2-\lambda_1)\lambda_1\lambda_2}{H^2}(\bar{\n}_2\bar{R}_{32}-\bar{\n}_1\bar{R}_{31}).
\end{align*}
To see the last inequality, we calculate as follows,
\begin{align*}
    &F^{kl,pq}\left({h_{ij}-\f{\abs{A}^2}{H}g_{ij}}\right)\n_ih_{kl}\n_jh_{pq}\\
    = & \f{\lambda_2(\lambda_1-\lambda_2)}HF^{kl,pq}\n_1h_{kl}\n_1h_{pq} +\f{\lambda_1(\lambda_2-\lambda_1)}HF^{kl,pq}\n_2h_{kl}\n_2h_{pq}.
\end{align*}
Note that $|F^{kl,pq}|\leq cC_0\sigma$ and from Lemma \ref{lem1.4}, 
\[|\nabla A|^2\leq 5|\n \mathring{A}|^2+4|\omega|^2\leq c(C_0^2+B_1^2+B_3^2)\sigma^{-8},\]
which implies the last inequality.
\end{proof}

We are now ready to prove an a priori estimate for the difference between
the principal curvatures.
\begin{prop}\label{prop1.8}
Suppose that the solution $\Sigma_t$ is contained in $\mathcal{B}_\sigma(B_1,B_2,B_3)$ for $t\in[0,T]$. Then there is an absolute constant c such that for $\sigma\geq c(C_0^2+B_1^2+B_2+B_3)$ we have the estimate
$$
|\mathring{A}|^2\leq cC_0^2\sigma^{-6}
$$
everywhere in $[0,T]$.
\end{prop}
\begin{proof}
 Consider the function $e=\f{|\mathring{A}|^2}{H^2}.$ Computing as in \cite{Huisken1987TheVP} we obtain from Lemma \ref{lem1.5} (2) (5) that
$$
\begin{aligned}
\frac{\partial }{\partial t}e=&\mathcal{L} e+\f2HF^{kl}\n_{l}H\n_ke-\frac{2}{H^{4}}|\nabla_{i}Hh_{kl}-\nabla_{i}h_{kl}H|_F^{2}+\red{\f{F^{kl}\mathring{h}_{kl}}{H}|\mathring{A}|^2}\\
&+\frac{2f}{H}e\bar{R}ic(\nu,\nu)+\frac{2f}{H^{3}}\big(|A|^{4}-H\:tr(A^{3})\big)+\frac{2f}{H^2}\bar{R}_{3i3j}\ch\\
& -\frac{4F^{kl}}{H^2}\left({h_{ij}-\f{\abs{A}^2}{H}g_{ij}}\right)(h_{jm}\bar{R}_{mkli}-h_{km}\bar{R}_{mjil})\\
&-\f{2F^{kl}}{H^2}\left({h_{ij}-\f{\abs{A}^2}{H}g_{ij}}\right)(\bar{\nabla}_{l}\bar{R}_{3jki}+\bar{\nabla}_{i}\bar{R}_{3klj})\\
& +\f2{H^2}\left(h_{ij}-\f{\abs{A}^2}{H}g_{ij}\right)(F^{kl,pq}\n_ih_{kl}\n_jh_{pq}),
\end{aligned}
$$
where 
\[
|\nabla_{i}Hh_{kl}-\nabla_{i}h_{kl}H|_F^2=F^{kl}\(h_{ij,k}h_{ij,l}H^2+H_kH_l\abs{A}^2-2h_{ij,k}h_{ij}H_lH\).
\]
Applying Lemma \ref{lem1.7}, and noting that 
$\bar{R}ic(\nu, \nu)$ is negative, we conclude that 
\[
\frac{d}{dt}e\leq \mathcal{L} e+\f2HF^{kl}\n_{l}H\n_ke-\f18\ma+cC_0|\mathring{A}|\sigma^{-3}.
\]
provided that $\sigma\geq c\left(C_0^2+B_1^2+B_2+B_3\right)$ is so large that Propositions \ref{2.1}, \ref{2.2}, \ref{3.4}, \ref{3.5} and Lemma \ref{lem1.7} all apply. Now suppose that $f$ reaches a value $D\sigma^{-4}$ for the first time at $(y_0,t_0)$, where $D$ is some positive number larger than the initial values of $e\sigma^4.$ Then at $(y_0,t_0)$ we have $(d/dt)e\geq0,\mathcal{L} e\leq0$ and $\nabla e=0$
such that
$$
0\leq-\frac18|\mathring{A}|^2+cC_0|\mathring{A}|\sigma^{-3}\quad\text{ at }(y_0,t_0).
$$
Since at that point $e=|\mathring{A}|^2/H^2=D\sigma^{-4}$ and since $\sigma^{-1}\leq H\leq3\sigma^{-1}$ we get
$$0\leq-\frac18D\sigma^{-6}+cC_0D^{1/2}\sigma^{-6},$$
which is a contradiction for $D\geq cC_0^2.$ Hence $e\leq cC_0^2\sigma^{-4}$, completing the proof of the proposition.
\end{proof}

 We write $A*B$ for any linear combination of contractions of $A$ and $B$ with the metric $g_{ij}.$ 
 
\begin{lem}\label{na}
    If $\Sigma_t$ is a solution of \eqref{flow} contained in $\mathcal{B}_\sigma(B_1, B_2,B_3)$, then
there is an absolute constant $c_1$ such that
\[
\pd t|\n\mathring{A}|^2 \leq\mathcal{L}|\n\mathring{A}|^2+c_1|\nabla\mathring{A}\mid^2\sigma^{-2}+c_1C_0|\nabla\mathring{A}\mid\sigma^{-6},
\]
provided $\sigma\geq c_1 (C_0^2+B_1^2+B_2)$.
\end{lem}

\begin{proof}
    Since $(d/dt)g_{ij}=2(f-F)h_{ij}$, the time derivative of the Christoffel
symbols is of the form $A*\nabla A$ and thus
$$
\begin{aligned}\frac{\partial}{\partial t}|\nabla\mathring{A}\:|^{2}
&=2\left(\frac{\partial}{\partial t}\nabla_k\ch\right)\nabla_k\ch\\
&=2\nabla_k\left(\frac{\partial}{\partial t}\ch\right)\nabla_k\ch+\nabla A*\nabla\mathring{A}\:*\mathring{A}\:*A.\end{aligned}
$$
From Lemma \ref{lem1.5} (3), we have
\begin{align*}
\frac{\partial}{\partial t}\ch =&  \mathcal{L} \ch+\mathring{A} *A*A+\mathring{A} *\bar{R}ic+ f(\bar{R}_{3i3j}+\frac{1}{2}\bar{R}ic(\nu,\nu))g_{ij})\\
& -F^{kl}\left(\bar{\nabla}_l\bar{R}_{3jki}+\bar{\nabla}_i\bar{R}_{3klj}-\frac{\bar{\nabla}_l\bar{R}_{3mkm}+\bar{\nabla}^m\bar{R}_{3klm}}2g_{ij}\right) \\
&+F^{kl,pq}\left(\nabla_ih_{pq}\nabla_jh_{kl}-\frac{\nabla_mh_{pq}\nabla^mh_{kl}}2g_{ij}\right).
\end{align*}
Hence we obtain 
\begin{align*}
    &\pd t|\n\mathring{A}|^2\\ =&2\nabla_k\ch\nabla_k\mathcal{L}\ch+\nabla\mathring{A}*\nabla A*\mathring{A} *A+\nabla\mathring{A} *\nabla\mathring{A} *A*A\\
    &+\nabla\mathring{A} *\nabla\mathring{A} *\bar{R}ic+\nabla\mathring{A}*\mathring{A} *\bar{\nabla}\bar{R}ic+\nabla\mathring{A} *\mathring{A}*A*\bar{R}ic\\
    & + 2f\nabla_k\ch\nabla_k\big(\bar{R}_{3i3j}+\frac{1}{2}\bar{R}ic(\nu,\nu)g_{ij}\big)\\
    & - 2\n_k\ch\left(\n_kF^{pq}(\bar{\n}_q\bar{R}_{3jpi}+\bar{\n}_i\bar{R}_{3pqj})+F^{pq}\n_k(\bar{\n}_q\bar{R}_{3jpi}+\bar{\n}_i\bar{R}_{3pqj})\right)\\
    & + 2\n_k\ch\left(\n_kF^{st,pq}\n_ih_{pq}\n_jh_{st}+2F^{st,pq}\n_k\n_ih_{st}\n_jh_{pq}\right).
\end{align*}
Now observe that 
\[
\begin{aligned}
&2\nabla_{k}(\mathcal{L}\ch)\nabla_{k}\ch\\
 =&2\mathcal{L}(\nabla_{k}\ch)\nabla_{k}\ch+2\n_k\ch\n_kF^{pq}\mathring{h}_{ij,pq}+\text{curvature terms} \\
=&\mathcal{L}|\n\mathring{A}|^2-2F^{pq}\mathring{h}_{ij,kp}\mathring{h}_{ij,kq}+2\n_k\ch\n_kF^{pq}\mathring{h}_{ij,pq}+\nabla\mathring{A} *\nabla\mathring{A} *\bar{R}ic \\
&+\nabla\mathring{A} *\mathring{A} *\bar{\nabla}\bar{R}ic+\nabla\mathring{A} *\mathring{A} *A*\bar{R}ic.
\end{aligned}
\]
Consequently, we get
\begin{align*}
    &\pd t|\n\mathring{A}|^2\\
    =&\mathcal{L}|\n\mathring{A}|^2-2F^{pq}\mathring{h}_{ij,kp}\mathring{h}_{ij,kq}+2\n_k\ch\n_kF^{pq}\mathring{h}_{ij,pq}+\nabla\mathring{A}*\nabla A*\mathring{A} *A\\
    &+\nabla\mathring{A} *\nabla\mathring{A} *A*A
    +\nabla\mathring{A} *\nabla\mathring{A} *\bar{R}ic+\nabla\mathring{A}*\mathring{A} *\bar{\nabla}\bar{R}ic\\
    & +\nabla\mathring{A} *\mathring{A}*A*\bar{R}ic
     + 2f\nabla_k\ch\nabla_k\big(\bar{R}_{3i3j}+\frac{1}{2}\bar{R}ic(\nu,\nu)g_{ij}\big)\\
    & - 2\n_k\ch\left(\n_kF^{pq}(\bar{\n}_q\bar{R}_{3jpi}+\bar{\n}_i\bar{R}_{3pqj})+F^{pq}\n_k(\bar{\n}_q\bar{R}_{3jpi}+\bar{\n}_i\bar{R}_{3pqj})\right)\\
    & + 2\n_k\ch\left(\n_kF^{st,pq}\n_ih_{pq}\n_jh_{st}+2F^{st,pq}\n_k\n_ih_{st}\n_jh_{pq}\right).
\end{align*}
To archive the desired estimate, we need the following lemma. 
\begin{lem}\label{lem1.9}
    Let $\Sigma_t$ be a smooth solution of \eqref{flow} contained in $\mathcal{B}_\sigma(B_1, B_2, B_3)$ for $t\in[0, T]$. Then there is an absolute constant $c$ such that  
    \[
    |d^kF|\leq c C_0^2\sigma^{k-1},
    \]
    for $\sigma\geq c(C_0^2+B^2_1+B_2)$. Here $d^kF$ is the k-th derivatives with respect to second fundamental form. 
\end{lem}
With this lemma in hand, and also observe that in view of Lemma \ref{lem1.4}, 
\begin{align*}
|\nabla A|^{2}=|\nabla\mathring{A} |^{2}+\frac{1}{2}|\nabla H|^{2}
&\leq5|\nabla\mathring{A} |^{2}+4\sum_{i}|\bar{R}ic(\nu,e_{i})|^{2}\\
&\leq5|\nabla\mathring{A} |^{2}+c (C_{0}^{2}+B_{1}^{2})\sigma^{-8}.
\end{align*}
Assuming then that $\sigma\geq c\left(C_{0}^{2}+B_{1}^{2}+B_{2}+B_{3}\right)$ is so large that Lemma \ref{lem1.7} applies, we get
$$
\begin{aligned}
&|\nabla A*\nabla\mathring{A}\:*A*\mathring{A}\:|\leq cB_{2}(C_{0}+B_{1})|\nabla\mathring{A}\:|\sigma^{-8},\\&|\nabla\mathring{A}\:*\mathring{A}\:*A*\bar{R}ic|\leq cC_{0}B_{2}|\nabla\mathring{A}\:|\sigma^{-7},\\&|\nabla\mathring{A}\:*\nabla\mathring{A}\:*A*A|\leq c|\nabla\mathring{A}\:|^{2}\sigma^{-2},\\&|\nabla\mathring{A}\:*\nabla\mathring{A}\:*\bar{R}ic|\leq cC_{0}|\nabla\mathring{A}\:|^{2}\sigma^{-3},\\&|\nabla\mathring{A}\:*\mathring{A}\:*\bar{\nabla}\bar{R}ic|\leq cC_{0}B_{2}|\n\mathring{A}\:|\sigma^{-7}.
\end{aligned}
$$
Furthermore, we can argue exactly as in the proof of Lemma \ref{lem1.7} to obtain the estimates
$$
\begin{vmatrix}
2f\nabla_k\ch\nabla_k(\bar{R}_{3i3j}+\frac{1}{2}\bar{R}ic(v,v)g_{ij})
\end{vmatrix}
\leq cC_0|\nabla\mathring{A}\:|\sigma^{-6},
$$
$$
\begin{vmatrix}
2\nabla_k\ch F^{pq}\nabla_k(\bar{\nabla}_q\bar{R}_{3jpi}+\bar{\nabla}_iR_{3pqj})
\end{vmatrix}
\leq
cC_0|\nabla\mathring{A}\:|\sigma^{-6}.
$$
When combined with Lemma \ref{lem1.9}, we can also obtain 
\[
|2\n_k\ch\n_kF^{pq}(\bar{\n}_q\bar{R}_{3jpi}+\bar{\n}_i\bar{R}_{3pqj})|
\leq c(C_0+B_1)|\n\mathring{A}|\sigma^{-7}.
\]
To deal with the 2nd derivatives, we observe that 
\begin{align*}
    2\n_k\ch\n_kF^{pq}\mathring{h}_{ij,pq} & \leq 2|d^2F||\n\mathring{A}||\n A||\n^2\mathring{A}|\\
    & \leq \eps|\n^2\mathring{A}|^2+\f1\eps|d^2F|^2|\n\mathring{A}|^2|\n A|^2,
\end{align*}
and
\begin{align*}
    2\n_k\ch\n_kF^{st,pq}\n_ih_{pq}\n_jh_{st} & \leq 
    |d^3F||\n\mathring{A}||\n A|^3\\
    & \leq cC_0^2|\n\mathring{A}|\sigma^{-10},
\end{align*}
and finally in view of Lemma \ref{lem1.4}
\begin{align*}
    4\n_k\ch F^{st,pq}\n_k\n_ih_{st}\n_jh_{pq} & \leq 4|\n\mathring{A}||d^2F||\n^2A||\n A|\\
    & \leq cC_0^2|\n\mathring{A}|(|\n^2\mathring{A}|+\abs{\n w})\sigma^{-3}\\
    & \leq \eps|\n^2\mathring{A}|^2+\f c\eps C_0^2|\n\mathring{A}|^2\sigma^{-6}.
\end{align*}

We can then derive the differential inequality for $|\n\mathring{A}|^2$, 
\begin{align*}
     \pd t|\n\mathring{A}|^2 \leq\mathcal{L}|\n\mathring{A}|^2+c|\nabla\mathring{A}\mid^2\sigma^{-2}+cC_0|\nabla\mathring{A}\mid\sigma^{-6}.
\end{align*}
\end{proof}

\begin{prop}\label{3.10}
    Suppose that the solution $\Sigma_t$ of \eqref{flow} is contained in $\mathcal{B}_\sigma(B_1,B_2, B_3)$ for $t\in[0,T]$. Then there is an absolute constant c such that for $\sigma \geq c( C_0^2+B_1^2+B_2+ B_3)$ the estimate
$$
|\nabla\mathring{A}\mid^2\leqq cC_0^2\sigma^{-8}
$$
holds everywhere in $[0,T]$.
\end{prop}
\begin{proof}
    From Lemma \ref{lem1.5} and the estimate of Lemma \ref{lem1.7}, we find that 
    \begin{align*}
    \pd t\ma  \leq & \mathcal{L}(\ma)-\f14|\n\mathring{A}|^2-2fH\ma+2F^{kl}h_{mk}h_{ml}\ma\\
    & +2f\mathring{h}_{ij}\bar{R}_{3i3j}-2F^{kl}\bar{R}_{3k3l}\ma-2F^{kl}\mathring{h}_{ij}(\bar{\n}_l\bar{R}_{3jki}+\bar{\n}_i\bar{R}_{3klj})\\
    & -4F^{kl}\mathring{h}_{ij}(h_{jm}\bar{R}_{mkli}-h_{km}\bar{R}_{mjil})+2F^{kl,pq}\ch\n_ih_{kl}\n_jh_{pq}\\
     \leq&  \mathcal{L}(\ma)-\f14|\n\mathring{A}|^2+cC_0|\mathring{A}|\sigma^{-5},
    \end{align*}
   for $\sigma\geq c(C_0^2+B_1^2+B_2)$. Now let $C_1$ be sufficiently large compared to the absolute constant $c_1$ in Lemma \ref{na}. Then
   \[
   \begin{aligned}
   \pd t(|\nabla\mathring{A}|^2+C_1|\mathring{A}|^2\sigma^{-2}) \leq&\mathcal{L}(|\nabla\mathring{A}|^2+C_1|\mathring{A}|^2\sigma^{-2})\\
   &-c_1|\nabla\mathring{A}|^2\sigma^{-2}+c_1C_0|\nabla\mathring{A} |\sigma^{-6}+cC_0^2\sigma^{-10},
   \end{aligned}
   \]
   for $\sigma\geq c(C_0^2+B_1^2+B_2+B_3)$. Using the parabolic maximum principle as before the desired estimate now follows from Proposition \ref{prop1.8}.
\end{proof}

From Proposition \ref{3.5}, \ref{prop1.8} and \ref{3.10}, we have established Theorem \ref{3.3}. Consequently, higher derivative estimates and long time existence follows easily.
 To prove the exponential convergence to a constant harmonic mean curvature surface $\Sigma_{\infty}$, we first need the lower bound of the operator $\mathcal{L}$.

 \begin{lem}\label{mu1}
     Suppose $\Sigma$ is a hypersurface of $N$ in $\mathcal{B}_{\sigma}(B_1, B_2, B_3)$. Then, for any function $u$ with $\int_{\Sigma}ud\mu=0$ satisfies the estimates
      \[ \int_{\Sigma}u\mathcal{L}ud\mu\leq  -\left( \frac{1}{2\sigma^2}-\frac{m}{\sigma^3} -c\sigma^{-4} \right)\int_{\Sigma}u^2d\mu, \]
      and
            \[ \int_{\Sigma}F^{ij}u_iu_jd\mu\geq  \left( \frac{1}{2\sigma^2}-\frac{m}{\sigma^3} -c\sigma^{-4} \right)\int_{\Sigma}u^2d\mu, \]
     provided $\sigma\geq c(C_0^2+B_1^2+B_2+B_3)$, where $c$ depends on $C_0, B_1, B_2, B_3$.
 \end{lem}

\begin{proof}
    By Lemma 3.13 in \cite{Huisken1996DefinitionOC}, the lowest eigenvalue $\mu_{Lap}$ of Laplace-Beltrami operator on $\Sigma$ satisfies
    \begin{equation}\label{le}
        \mu_{Lap}\geq \frac{2}{\sigma^2}-\frac{4m}{\sigma^3}-cC_0B_1\sigma^{-4}.
    \end{equation}  
Computing directly , we have pointwisely
\[ F^{11}=\frac{\lambda_2^2}{(\lambda_1+\lambda_2)^2}, \quad
 F^{22}=\frac{\lambda_1^2}{(\lambda_1+\lambda_2)^2}.\]
Then we can decompose the operator $\mathcal{L}$ by, for any function $u\in C^{\infty}(\Sigma)$,
\begin{align*}
\mathcal{L}u & =F^{kl}u_{kl}=F^{11}u_{11}+F^{22}u_{22}\\
&=\frac{1}{4} \Delta u+\frac{1}{H^{2}}\left(\mathring{h}^{ij}-H{g}^{ij}\right)\mathring{h}_{jk} u_{ki}.
\end{align*}
Integrating by parts gives
\begin{align*}
 & \int_{\Sigma}\frac{u}{H^{2}} \left(\mathring{h}^{ij}-H{g}^{ij}\right)\mathring{h}_{jk} u_{ki}d\mu\\
=&-\int_{\Sigma}\f1{H^{2}}\left(\mathring{h}_{ij}-Hg_{ij}\right)\mathring{h}_{jk} u_ku_id\mu+\int_{\Sigma}\f{2u}{H^3}\left(\mathring{h}_{ij}-Hg_{ij}\right)\mathring{h}_{jk}H_ku_id\mu\\
&- \int_{\Sigma}\f u{H^2}\left(\mathring{h}_{ij,k}-H_kg_{ij}\right)\mathring{h}_{jk}u_id\mu-\int_{\Sigma}\f u{H^2}\left(\mathring{h}_{ij}-Hg_{ij}\right)\mathring{h}_{jk,k}u_id\mu\\
 \leq & c\int_{\Sigma}[\sigma^{-2}|\n u|^2 +|u||\n u| \sigma^{-3}]d\mu\\
\leq& c\left(\sigma^{-2}\int_{\Sigma}|\n u|^2d\mu+\sigma^{-4}\int_{\Sigma}u^2d\mu\right).
\end{align*}
By \eqref{le}, we get
\begin{equation}\label{e-cal-L}
    \begin{split}
        &\int_{\Sigma}u\mathcal{L}ud\mu \\
        \leq& -\frac{1}{4}\int_{\Sigma}|\n u|^2d\mu +c\left(\sigma^{-2}\int_{\Sigma}|\n u|^2d\mu+\sigma^{-4}\int_{\Sigma}u^2d\mu\right)\\
        = & -\left(\frac{1}{4}-c\sigma^{-2}\right)\int_{\Sigma}{|\n u|^2d\mu}+c\sigma^{-4}\int_{\Sigma}u^2d\mu\\
        \leq & -\left( \frac{1}{2\sigma^2}-\frac{m}{\sigma^3} -c\sigma^{-4} \right)\int_{\Sigma}u^2d\mu.
    \end{split}
\end{equation}
This proves the first inequality.

To show the second inequality, we use (\ref{e-cal-L}) and the fact that $|F^{ij}_{,j}|\leq c\sigma^{-3}$ to compute
\begin{eqnarray*}
\int_{\Sigma}F^{ij}u_iu_jd\mu
&=&-\int_{\Sigma}u\mathcal{L}ud\mu-\int_{\Sigma}F^{ij}_{,j}u_iud\mu\\
&\geq & \left(\frac{1}{4}-c\sigma^{-2}\right)\int_{\Sigma}{|\n u|^2d\mu}-c\sigma^{-4}\int_{\Sigma}u^2d\mu-c\sigma^{-3}\int_{\Sigma}|\nabla u||u|d\mu\\
&\geq & \left(\frac{1}{4}-c'\sigma^{-2}\right)\int_{\Sigma}{|\n u|^2d\mu}-c'\sigma^{-4}\int_{\Sigma}u^2d\mu\\
        &\geq & \left( \frac{1}{2\sigma^2}-\frac{m}{\sigma^3} -c''\sigma^{-4} \right)\int_{\Sigma}u^2d\mu.
\end{eqnarray*}
\end{proof}

From (1) of Lemma \ref{lem1.7}, we have $|F-f|\leq c\sigma^{-2}$, where $c$ depends on $C_0, B_2, B_3$. Actually, we have better decay estimate of $|F-f|$.

\begin{lem}\label{sig3}
    Suppose $\Sigma_t$ is a solution of \eqref{flow} contained in $\mathcal{B}_{\sigma}(B_1, B_2, B_3)$.
    Then, for all $t\geq0$, 
    \[ |F-f|\leq c\sigma^{-3}, \]
    where $c$ depends on $C_0, B_1, B_2, B_3$.
\end{lem}
\begin{proof}
    Since $F=\frac{H}{2}-\frac{|A|^{2}}{2 H}$, we can calculate directly that
    \[ \nabla F=\frac{\nabla H}{2}-\frac{\nabla |A|^{2}}{2H}+\frac{\vert A\vert^{2}}{2H^{2}} \nabla H=\frac{1}{2}\left(1+\frac{\vert A\vert^{2}}{H^{2}}\right) \nabla H-\frac{\vert A\vert \nabla|A|}{H}. \]
    By Lemma \ref{lem1.4}, we know that
    \[ |\nabla H|^{2} \leq 8|\nabla\mathring{A}|^{2}+8\omega^{2} \leq c\sigma^{-8}, \]
    \[ |\nabla |A||^{2} \leq|\nabla A|^{2} \leq 5|\nabla\mathring{A}|^{2}+4\omega^{2} \leq c\sigma^{-8}. \]
Then, 
\[ |\n F|^2\leq c\sigma^{-8}. \]
Therefore, 
\[ |F-f|\leq diam(\Sigma_t)\max_{\Sigma_t}|\n F|\leq c\sigma^{-3}, \]
    where where $c$ depends on $C_0, B_1, B_2, B_3$. Here $diam(\Sigma_t)$ is the intrinic diameter, which by the 1st variation formula is bounded by the Willmore energy.
\end{proof}

Next, we show the surfaces $\Sigma_t$ converge exponentially to a constant harmonic mean curvature surface $\Sigma_{\sigma}$. 
From the Appendix, we have
\begin{equation}\label{pf}
     \pd tF=\mathcal{L}F-(f-F)(F^{kl}h_{mk}h_{lm}-F^{kl}\bar{R}_{3k3l}).
\end{equation}
Then, in view of the fact that $\int_{\Sigma_t}(f-F)d\mu_t=0$, we obtain
\begin{align*}
&\frac{d}{dt}\int_{\Sigma_t}(f-F)^2d\mu_{t}\\
 =&\int_{\Sigma_t}[-2(f-F)\mathcal{L}F+2(f-F)^2(F^{kl}h_{mk}h_{ml}-F^{kl}\bar{R}_{3k3l})+H(f-F)^3]d\mu_{t}.
 \end{align*}
 By Proposition \ref{2.2} and \ref{3.4}, we can calculate
 \begin{align}\label{f2}
     &F^{kl}h_{mk}h_{ml}=F^{11}\lambda_1^2+F^{22}\lambda_2^2=2F^2 \\
     =& 2\left( \frac{H}{4}-\frac{|\mathring{A}|^2}{2H} \right)^2
     =\frac{1}{2\sigma^2}-\frac{2m}{\sigma^3}+O(\sigma^{-4}),\notag
 \end{align}
 and 
\begin{align}
&F^{kl}\bar{R}_{3k3l}=F^{11}\bar{R}_{3131}+F^{22}\bar{R}_{3232}\notag\\
    =& \frac{\lambda_2^2-\lambda_1^2}{(\lambda_1+\lambda_2)^2}\bar{R}_{3131}-\frac{\lambda_1^2}{(\lambda_1+\lambda_2)^2}\bar{R}ic(\nu, \nu)\notag\\
    =&\frac{m}{2\sigma^3}+O(\sigma^{-4}),\label{r3i3j}
\end{align}
Using Lemma \ref{mu1}, we compute
\begin{align*}
&\frac{d}{dt}\int_{\Sigma_t}(f-F)^2d\mu_{t}\\
\leq & -\left( \frac{1}{\sigma^2}-\frac{2m}{\sigma^3}-c\sigma^{-4} \right)\int_{\Sigma_{t}}(F-f)^{2}d\mu_t\\
& +2\left(\f1{2\sigma^{2}}-\f {2m}{\sigma^3}-\frac{m}{2\sigma^{3}}+\f c{\sigma^4}\right)\int_{\Sigma_{t}}(F-f)^{2}d\mu_t
 +\int_{\Sigma_{t}}H(f-F)^{3}d\mu_t \\
\leq &-\left( \frac{1}{\sigma^2}-\frac{2m}{\sigma^3}-c\sigma^{-4} \right)\int_{\Sigma_{t}}(F-f)^{2}d\mu_t+\left( \frac{1}{\sigma^2}-\frac{5m}{\sigma^3}+c\sigma^{-4} \right)\int_{\Sigma_{t}}(F-f)^{2}d\mu_t\\
& +\int_{\Sigma_{t}}H(f-F)^{3}d\mu_t \\
\leq& \left(-\frac{3m}{\sigma^{3}}+c\sigma^{-4}\right)\int_{\Sigma_t}(F-f)^{2}d\mu_t-\int_{\Sigma_t}H(F-f)^3d\mu_t.
\end{align*}
 By Lemma \ref{sig3}, we have $|H(F-f)|\leq c\sigma^{-4}$.
Therefore, we obtain
\begin{equation}\label{exp}
    \frac{d}{dt}\int_{\Sigma_t}(f-F)^2d\mu_{t}\leq (-\frac{3m}{\sigma^{3}}+c\sigma^{-4})\int_{\Sigma_t}(F-f)^{2}d\mu_t,
\end{equation}  
which implies that
\begin{equation}\label{expo}
    \int_{\Sigma_t}(f-F)^2d\mu_{t}\leq e^{-\frac{2mt}{\sigma^3}}\int_{\Sigma_0}(f-F)^2d\mu_{0},
\end{equation}
for $\sigma\geq\sigma_0$.
Exponential convergence follows by standard interpolation inequalities,  completing the proof of Theorem \ref{main}.
\hfill $\boxed{}$

\vspace{.2in}

\section{The Foliation}
In this section, we prove the existence of the foliation by constant harmonic mean curvature surfaces. We begin by providing a sufficient condition that guarantees a family of surfaces form a foliation. 
\begin{lem}\label{fol}
    Let $\Sigma_\sigma=\phi^\sigma(\mathbb{S}^2)\hookrightarrow N$ be a family of surfaces parametrized by $\sigma\in[\sigma_0,\infty)$. Define the lapse function by 
    $u=\<\pd\sigma\phi^\sigma,\nu\>$, where $\nu$ is the unit outer normal to $\Sigma_\sigma$. If furthermore $u>c>0$, then this family of surfaces will foliate the ambient space in $N\setminus B_{2\sigma_0}$. 
\end{lem}
\begin{proof}
    We will prove that the union of this family of surfaces will pass through any points in $N\backslash B_{2\sigma_0}$. In other words, for any $x\in N\setminus B_{2\sigma_0}$, there exists some surface $\Sigma_\sigma$, for $\sigma>\sigma_0$ such that $x\in \Sigma_{\sigma}$. It is important to note that it is not true generally that the radial exponential map is a global diffeomorphism. And thus rather than proving the global diffeomorphism,  we will prove that there exists a positive constant $\delta_0$ such that for any $i\in\mathbb{N}$, the map defined by 
    \[
    F_i:[i\delta_0,(i+1)\delta_0]\times\pa B_{2\sigma_0+i\delta_0}\longrightarrow N\setminus B_{2\sigma_0}
    \]
    \[
    (t,p)\longrightarrow \exp_p^N(t\f{\pa \phi^\sigma}{\pa \sigma})
    \]
    is a local diffeomorphism. Indeed, we can take 
    \[
\delta_0=\inf_{p}\norm{(\exp_p^N)^{-1}}=\inf_{p,X_p\in(\exp_p^N)^{-1}(W_{p,i})}\norm{X_p},
\]
where $W_{p,i}=\mathrm{Img}(F_i)$ is the tubular neighbourhood of the surface $\pa B_{2\sigma_0+i\delta_0}$.
    \begin{claim}
        $F_i$ is a local diffeomorphism, and the constant $\delta_0$ must be positive.
    \end{claim}
    \begin{proof}[Proof of claim]
      The first statement can be seen as an application of the inverse function theorem. Hence we calculate the differential of the map $F_i$ as for any fixed $(t_0,p_0)$ we have
      \begin{align*}
      dF_i|_{(0,p_0)}(t,v) & =\f d{ds}|_{s=0}\left[\exp_{\gamma(s)}((t+s)\f{\pa \phi^\sigma}{\pa \sigma})\right]\\
      & =(d\exp_p)_{t\f{\pa \phi^\sigma}{\pa \sigma}}(\f{\pa \phi^\sigma}{\pa \sigma})+U(0,t),
      \end{align*}
      where $\gamma(s)$ is a curve starting at $p$ with $\dot{\gamma}(0)=v_p$ and $U(0,t)$ is the variational field along the curve $\xi(t)=\exp_p(t\f{\pa \phi^\sigma}{\pa \sigma})$. We then find that at $s=0$, $U(0,0)=\dot{\gamma}(0)=v$, thus 
      \[
      dF_i|_{(0,p_0)}(0,v)=v_p+\f{\pa \phi^\sigma}{\pa \sigma}.
      \]
      It follows that $dF_i$ is a local diffeomorphism. And the inverse function theorem gives some positive constant $\delta_i$  such that 
      \[
      G_i:W_i\longrightarrow [-\delta_i,\delta_i]\times\pa B_{2\sigma_0+\delta_{i-1}}
      \]
      is a diffeomorphism. We now prove that there exists some positive constant $\delta_0$ such that $2\delta_i\geq\delta_0$. Fix some $p\in\pa B_{2\sigma_0+\delta_i}$, let $l_i$ be the length of the curve $\xi(t)=\exp_p(t\f{\pa\phi^\sigma}{\pa\sigma})$ and $\theta_i$ the angle between $\f{\pa\phi^\sigma}{\pa\sigma}$ and $\nu$ at $p$, then we have 
      \[
      \delta_i\geq dist(q_i,\pa B_{2\sigma_0+\delta_i})\geq \f12l_i\cos\theta_i\geq \f12c.
      \]
      Here $q_i=\exp_p(l_i\f{\pa\phi^\sigma}{\pa\sigma})$ is the end point of $\xi(t)$. Here $l_i$ has a lower bound, which is due to the positivity of the injectivity radius. 
      \end{proof}
    Once the claim is established, we can then glue all the local diffeomorphisms together. It is noted that possibly the gluing part of $F_i$ and $F_{i+1}$ would be non-smooth, but this issue can be addressed by slightly modifying the map, this completes the proof. 
    
\end{proof}

Consider the smooth operator $\mathscr{F}: C^3(S^2, N)\rightarrow C^1(S^2)$ which assigns to each embedding $\phi: S^2\rightarrow N$ the harmonic mean curvature $\mathscr{F}(\phi)$ of the surface $\Sigma=\phi(S^2)$. 
Given a variation vector field $V$ on a constant harmonic mean curvature surface $\Sigma$, we could compute directly that the first variation of the $\mathscr{F}$ operator at $\Sigma$ in direction $V$ is given by
\[
d\mathscr{F}(\phi)\cdot V=-\left(\mathcal{L}+F^{ij}h_{jk}h_{ik}-F^{ij}\bar{R}_{3i3j}\right)\<V,\nu\>,
\]
where $\phi:S^2\rightarrow{N}$ is the embedding of the constant harmonic mean curvature surface.
From \eqref{f2}, we know $F^{ij}h_{jk}h_{ik}=2F^2$.
Define the linearized harmonic mean curvature operator $L$ on $\Sigma$
\[
L=-\left(\mathcal{L}+2F^2-F^{ij}\bar{R}_{3i3j}\right).
\]
Since the operator $L$ is not self-adjoint, many useful tools in functional analysis and PDE cannot be applied to it. To proceed further, we consider the adjoint operator $L^*$ of $L$. By direct computation, we see that
\begin{equation}\label{e-adjoint-L}
    L^*u=Lu-2F^{ij}_{,i}u_j-F^{ij}_{,ij}u.
\end{equation}
Then the operator
\begin{eqnarray}\label{e-S}
    Su&:=&\frac{1}{2}(L+L^*)u=Lu-F^{ij}_{,i}u_j-\frac{1}{2}F^{ij}_{,ij}u\nonumber\\
    &=& -\left(\mathcal{L}u+2F^2u-F^{ij}\bar{R}_{3i3j}u+F^{ij}_{,i}u_j+\frac{1}{2}F^{ij}_{,ij}u\right)\nonumber\\
    &=& -(F^{ij}u_j)_{,i}-\left(2F^2-F^{ij}\bar{R}_{3i3j}+\frac{1}{2}F^{ij}_{,ij}\right)u
\end{eqnarray}
is a self-adjoint operator.

\iffalse
\begin{defn}
    A smooth hypersurface $\Sigma$ in $N$ is called strictly harmonic stable if the linearized harmonic mean curvature operator $L$ satisfies 
    \[\mu_0:=\inf \left\{\int_{\Sigma}uLud\mu: ||u||_{L^2}=1, \int_{\Sigma}ud\mu=0\right\}>0.\]
\end{defn}
\fi

Denote 
    \[\mu_0:=\inf \left\{\int_{\Sigma}uSud\mu: ||u||_{L^2}=1, \int_{\Sigma}ud\mu=0\right\}>0.\]
Then we have the following lower bound for $\mu_0$:

\begin{lem}\label{strc}
    Let $\Sigma_{\sigma}$ be the constant harmonic mean curvature hypersurface constructed in Theorem \ref{main}. Then for $\sigma\geq\sigma_0$, we have
    %$\Sigma_{\sigma}$ is strictly harmonic stable and the lowest eigenvalue
    \[ \mu_0\geq \frac{3m}{2}\sigma^{-3}-c\sigma^{-4}. \]
\end{lem}

\begin{proof}
We denote $\Sigma_{\sigma}$ by $\Sigma$ for any $\sigma\geq\sigma_0$.
By \eqref{f2} and \eqref{r3i3j}, we see that 
\begin{equation}\label{fr}
 2F^2-F^{ij}\bar{R}_{3i3j}=\frac{1}{2\sigma^2}-\frac{5m}{2\sigma^3}+O(\sigma^{-4}).
 \end{equation} 
Using Lemma \ref{lem1.9} and Proposition \ref{3.10}, we also have
\begin{equation}\label{e-F-ij}
F^{ij}_{,ij}=O(\sigma^{-4}).
 \end{equation} 
For any $u$ with $||u||_{L^2}=1, \int_{\Sigma}ud\mu=0$, by Lemma \ref{mu1}, we have
\begin{equation*}
    \begin{split}
        \int_{\Sigma}uSud\mu=& \int_{\Sigma}\left[-u\mathcal{L}u-(2F^2-F^{ij}\bar{R}_{3i3j})u^2-\left(F^{ij}_{,i}u_ju+\frac{1}{2}F^{ij}_{,ij}u^2\right)\right]d\mu\\
        \geq& \int_{\Sigma}\left[-u\mathcal{L}u-(2F^2-F^{ij}\bar{R}_{3i3j})u^2\right]d\mu-c\sigma^{-4}\\
        \geq& \left(\frac{1}{2\sigma^2}-\frac{m}{\sigma^3}+c\sigma^{-4}\right)-\left(\frac{1}{2\sigma^2}-\frac{5m}{2\sigma^3}+c\sigma^{-4}\right)\\
        =&\frac{3m}{2\sigma^3}-c\sigma^{-4}.
    \end{split}
\end{equation*}
    
\end{proof}

\iffalse
\red{The next lemma is essentially to show that the linearized operator is invertible and the decay estimate. However, the proof employ the eigenfunction decomposition, which is not valid for a non-symmetric operator( here is the case, since our linearized operator is not symmetric, although it can be seen as a perturbation of the Laplacian operator).}  We have to employ another method to give the desired result. (However, yes there is another "but", \blue{the eigenvalues are actually "larger" than the first eigenvalue, in the sense that the real part of the eigenvalue is larger than the first eigenvalue, this fact indeed imply that $L$ is invertible and also a suitable estimate of the norm of the inverse.} 

\fi

\begin{lem}\label{inv}
   Let $\Sigma_{\sigma}$ be the constant harmonic mean curvature hypersurface constructed in Theorem \ref{main}. For $\sigma\geq\sigma_0$, the operator $S$ is invertible, and $|S^{-1}|\leq cm^{-1}\sigma^{3}$.
\end{lem}

\begin{proof}
     We denote $\Sigma_{\sigma}$ by $\Sigma$ for any $\sigma\geq\sigma_0$. Let $\eta_0$ be the lowest eigenvalue of $S$ without constraints, i.e.
    \begin{align*}
        \eta_0=\inf_{\{u: ||u||_{L^2}=1\}}\int_{\Sigma}uSud\mu=\inf_{\{u: ||u||_{L^2}=1\}}\int_{\Sigma}uLud\mu. 
    \end{align*}
    By Lemma \ref{mu1}, for $u$ with $||u||_{L^2}=1$ and $\int_{\Sigma}ud\mu=0$, there holds, for $\sigma\geq\sigma_0 $
    \begin{align*}
        &\int_{\Sigma}u\mathcal{L}ud\mu\\
        \leq & -\left(\frac{1}{4}-c\sigma^{-2}\right)\int_{\Sigma}{|\n u|^2d\mu}+c\sigma^{-4}\\
        \leq& c\sigma^{-4}.
    \end{align*}
   By \eqref{fr}, we have
   \[ \eta_0\geq-\frac{1}{2\sigma^2}+\frac{5m}{2\sigma^3}-c\sigma^{-4}.  \]
    On the other hand, if we replace $u$ by a constant, we obtain the reverse inequality. Hence, 
    \begin{equation*}
        \eta_0=-\frac{1}{2\sigma^2}+\frac{5m}{2\sigma^3}+O(\sigma^{-4}).
    \end{equation*}
    Let $h_0$ be the corresponding eigenfunction of $\eta_0$
    \[ Sh_0=\eta_0h_0. \]
    Let $\bar{h}_0=\fint_{\Sigma}h_0d\mu$ be the mean value of $h_0$.
    Multiplying the above identity with $(h_0-\bar{h}_0)$ and integrating it over $\Sigma$ to obtain
    \begin{align*}
        & -\int_{\Sigma}(h_0-\bar{h}_0)\mathcal{L}(h_0-\bar{h}_0)d\mu\\
        =& \int_{\Sigma}\left(\eta_0+2F^2-F^{ij}\bar{R}_{3i3j}+\frac{1}{2}F^{ij}_{,ij}\right)(h_0-\bar{h}_0)^2d\mu\\
        &+\int_{\Sigma}\left(\eta_0+2F^2-F^{ij}\bar{R}_{3i3j}+\frac{1}{2}F^{ij}_{,ij}\right)\bar{h}_0(h_0-\bar{h}_0)d\mu\\
        &+\int_{\Sigma}F^{ij}_{,i}h_{0,j}(h_0-\bar{h}_0)d\mu\\
        =& \int_{\Sigma}\left(\eta_0+2F^2-F^{ij}\bar{R}_{3i3j}\right)(h_0-\bar{h}_0)^2d\mu\\
        &+\int_{\Sigma}\left(\eta_0+2F^2-F^{ij}\bar{R}_{3i3j}+\frac{1}{2}F^{ij}_{,ij}\right)\bar{h}_0(h_0-\bar{h}_0)d\mu.
    \end{align*}
    From Lemma \ref{mu1}, the left hand side is bounded below by
    \[ -\int_{\Sigma}(h_0-\bar{h}_0)\mathcal{L}(h_0-\bar{h}_0)d\mu\geq \left( \frac{1}{2\sigma^2}-\frac{m}{\sigma^3} -c\sigma^{-4} \right)\int_{\Sigma}(h_0-\bar{h}_0)^2d\mu. \]
    Together with 
    \[ \eta_0+2F^2-F^{ij}\bar{R}_{3i3j}=O(\sigma^{-4}) \]
    and (\ref{e-F-ij}), we obtain, for $\sigma\geq\sigma_0$,  
    \begin{equation}\label{h0}
        ||h_0-\bar{h}_0||_{L^2}\leq c\sigma^{-2}|\bar{h}_0||\Sigma|^{\frac{1}{2}}.
    \end{equation}
    In particular, $\bar{h}_0\not=0$. Let $\eta_1$ be the next eigenvalue of $S$ with corresponding eigenfunction $h_1$.
    \begin{comment}
        here in general, $\eta_1$ is not a real number. But we have the following lemma.
    \begin{lem}
    Principal eigenvalue for non-symmetric elliptic operators,
    \begin{itemize}
        \item i)There exists a real eigenvalue $\lambda_1$ for the operator L,  taken with zero mean value integral, such that if $\lambda\in\mathbb{C}$ is any other eigenvalue,  we have
$$\mathrm{Re}(\lambda)\geq\lambda_1.$$
\item ii) There exists a corresponding eigenfunction $w_1$, which is positive within U.
\item iii) The eigenvalue $\lambda _1$ is simple; that is,  if u is any solution,  then u is a multiple of $w_1$.
  \end{itemize}
    \end{lem}
\end{comment}
    Let $\bar{h}_1=|\Sigma|^{-1}\int_{\Sigma}h_1d\mu$ be the mean value of $h_1$. 
    Note that
    \[ 0=\int_{\Sigma}h_0h_1d\mu=\int_{\Sigma}(h_0-\bar{h}_0)(h_1-\bar{h}_1)d\mu+\int_{\Sigma}\bar{h}_0h_1d\mu. \]
    By H\"older inequality, we get
    \[ \left|\int_{\Sigma}h_1d\mu\right|\leq |\bar{h}_0|^{-1}||h_0-\bar{h}_0||_{L^2}||h_1-\bar{h}_1||_{L^2}. \]
    Then, by \eqref{h0}, there holds
    \begin{equation}\label{h1}
        |\bar{h}_1|\leq c\sigma^{-2}|\Sigma|^{-\frac{1}{2}}||h_1-\bar{h}_1||_{L^2}.
    \end{equation}
Multiplying $Sh_1=\eta_1 h_1$ with $(h_1-\bar{h}_1)$ and integrating it over $\Sigma$,  we obtain by Lemma \ref{strc} and \eqref{h1},
\begin{align*}
&\left(\frac{3m}{2\sigma^3}-c\sigma^{-4}\right)\int_{\Sigma}(h_1-\bar{h}_1)^2d\mu\\
    \leq& \int_{\Sigma}(h_1-\bar{h}_1)S(h_1-\bar{h}_1)d\mu\\
    =& \eta_1\int_{\Sigma}(h_1-\bar{h}_1)^2d\mu+\bar{h}_1\int_{\Sigma}(h_1-\bar{h}_1)\left(2F^2-F^{kl}\bar{R}_{3k3l}+\frac{1}{2}F^{ij}_{,ij}\right)d\mu\\
    \leq &\eta_1\int_{\Sigma}(h_1-\bar{h}_1)^2d\mu
    +c\sigma^{-4}|\bar{h}_1|\int_{\Sigma}|h_1-\bar{h}_1|d\mu\\
    \leq &\eta_1\int_{\Sigma}(h_1-\bar{h}_1)^2d\mu+c\sigma^{-6}\int_{\Sigma}(h_1-\bar{h}_1)^2d\mu.
\end{align*}
Therefore, for $\sigma\geq\sigma_0$, 
\[ \eta_1\geq cm\sigma^{-3}. \]
Hence, we show that $S$ is invertible, and $|S^{-1}|\leq cm^{-1}\sigma^3$.
\end{proof}

%The constant harmonic mean curvature surfaces constructed in Theorem \ref{main} is strictly harmonic stable and $L$ is invertible. 
We are now ready to use Lemma \ref{fol} to show that $\{\Sigma_{\sigma}\}$ form a foliation. 
%Furthermore, $\Sigma_{\sigma}$ is actually unique in $\mathcal{B}_{\sigma}(B_1, B_2, B_3)$.

\begin{thm}\label{local}
    There is $\sigma_0$ depending only on $C_0$ such that for all $\sigma\geq\sigma_0$, the constant harmonic mean curvature surfaces $\Sigma_{\sigma}$ constructed in Theorem \ref{main} constitute a proper foliation of $N\backslash B_{\sigma_0}(0)$. 
    \iffalse Moreover, given constants $B_1, B_2, B_3$, we can choose $\sigma\geq \sigma_0\geq c(C_0^2+B_1^2+B_2+B_3)$ such that $\Sigma_{\sigma}$ is the only surface with constant harmonic mean curvature $f_{\sigma}$ contained in $\mathcal{B}_{\sigma}(B_1, B_2, B_3)$.
    \fi
\end{thm}

\begin{proof}
    Let $\Sigma_{\sigma}$ be the family of constant harmonic mean curvature surfaces constructed in Theorem \ref{main} with $\sigma\geq\sigma_0$, and $\phi^{\sigma}: S^2\rightarrow N$ be the embedding for each $\sigma$.

%Since $d\mathscr{F}=L$ is invertible by Lemma \ref{inv}, $\mathscr{F}$ is a local diffeomorphism by the inverse function theorem.
    For any $\sigma_2>\sigma_1\geq\sigma_0$ such that $\Sigma_{\sigma_i}\in \mathcal{B}_{\sigma}(B_1, B_2, B_3), i=1,2$, $\Sigma_{\sigma_2}$ can be represented by a normal variation over $\Sigma_{\sigma_1}$ of the form
    \[ \phi^{\sigma_2}=\phi^{\sigma_1}+u(p)\nu(p), \quad p\in S^2.
    \]
     We will show that $u$ has a sign; in particular, $u$ cannot be zero.
    In the following, we denote $\Sigma_{\sigma}$ by $\Sigma$.
    
Let $V=\phi^{\sigma_2}-\phi^{\sigma_1}$.
    By the Taylor theorem, there is some $\sigma\in(\sigma_1, \sigma_2)$ such that
    \begin{eqnarray}
        \mathscr{F}(\phi^{\sigma_2})&=&\mathscr{F}(\phi^{\sigma_1})+Lu+\frac{1}{2}d^2\mathscr{F}(\phi^{\sigma})(V, V)\nonumber\\
    &=&\mathscr{F}(\phi^{\sigma_1})+Su+(L-S)u+\frac{1}{2}d^2\mathscr{F}(\phi^{\sigma})(V, V),
    \end{eqnarray}
    where $\mathscr{F}(\phi^{\sigma_2})$ and $\mathscr{F}(\phi^{\sigma_1})$ are constants.
    Note that the second variation of the harmonic mean curvature operator involves the second fundamental form and higher derivatives of the metric, this leads to
    \begin{equation}\label{ddf}
        ||d^2\mathscr{F}(V,V)||\leq c\sigma^{-3}||V||^2\leq c\sigma^{-3}||u||_{C^{2}}.
    \end{equation}  
    We also have from the definition of $S$ that
       \begin{equation}\label{e-S-L}
        |(L-S)u|=\frac{|(L-L^*)u|}{2}=\left|F^{ij}_{,i}u_j+\frac{1}{2}F^{ij}_{,ij}u\right|\leq c\sigma^{-3}||u||_{C^{1}}.
    \end{equation}  
Hence, $u$ satisfies the following elliptic equation
\begin{equation}\label{lu}
    Su=\mathscr{F}(\phi^{\sigma_2})-\mathscr{F}(\phi^{\sigma_1})+E_1,
\end{equation}  
with the error term satisfying $|E_1|\leq c\sigma^{-3}||u||_{C^{2}}$.
We will employ Huang's approach of Theorem 3.9 in \cite{Huang}.
By integrating \eqref{lu} over $S^2$, 
\[ \mathscr{F}(\phi^{\sigma_2})-\mathscr{F}(\phi^{\sigma_1})=
        \frac{1}{|S^2|}\int_{S^2}Sud\mu-\frac{1}{|S^2|}\int_{S^2}E_1d\mu. \]
By the definition of $S$, we get
\begin{equation*}
    \begin{split}
        &\frac{1}{|S^2|}\int_{S^2}Sud\mu\\
        =& -\frac{1}{|S^2|}\int_{S^2}\left(\mathcal{L}u+2F^2u-F^{ij}\bar{R}_{3i3j}u+F^{ij}_{,i}u_j+\frac{1}{2}F^{ij}_{,ij}u\right)d\mu\\
        =& -\frac{1}{|S^2|}\int_{S^2}\left(\mathcal{L}u+2F^2u-F^{ij}\bar{R}_{3i3j}u-\frac{1}{2}F^{ij}_{,ij}u\right)d\mu\\
        =& \frac{1}{|S^2|}\int_{S^2}\left[-\frac{1}{4} \Delta u-\frac{1}{H^{2}}\left(\mathring{h}^{ij}-H{g}^{ij}\right)\mathring{h}_{jk} u_{ki}\right]d\mu-\left(\frac{1}{2\sigma^2}-\frac{5m}{2\sigma^3}+O(\sigma^{-4})\right)\bar{u}\\
        \leq & c\sigma^{-3}||u||_{C^1}-\left(\frac{1}{2\sigma^2}-\frac{5m}{2\sigma^3}+O(\sigma^{-4})\right)\bar{u}, 
    \end{split}
\end{equation*}
where $\bar{u}=\frac{1}{|S^2|}\int_{S^2}ud\mu$ is the mean value of $u$ and the last inequality is obtained by integration by parts.
It follows,
\begin{equation}\label{fphi}
    \mathscr{F}(\phi^{\sigma_2})-\mathscr{F}(\phi^{\sigma_1})=-\frac{1}{2\sigma^2}\bar{u}+E_2,
\end{equation}
where 
\[ |E_2|\leq c\sigma^{-3}||u||_{C^1}+c\sigma^{-3}||u||_{C^2}. \]
Now we decompose $u=h_0+u_0$, where $h_0$ is the lowest eigenfunction of $S$ and $\int_{S^2}h_0u_0d\mu=0$.

Due to De Giorgi-Nash-Moser iteration and \eqref{h0}, there holds the estimate
\[ \sup|h_0-\bar{h}_0|\leq c\sigma^{-2}|\bar{h}_0|. \]
We note that $(h_0-\bar{h}_0)$ satisfies the following equation
\[ S(h_0-\bar{h}_0)=\eta_0(h_0-\bar{h}_0)+\left(\eta_0+2F^2-F^{ij}\bar{R}_{3i3j}+\frac{1}{2}F^{ij}_{,ij}\right)\bar{h}_0. \]
And hence we obtain 
\begin{equation}\label{hal}
\begin{split}
    &||h_0-\bar{h}_0||_{C^{0, \alpha}}\\
    \leq& c\left(\sigma^{-\alpha}\sup_{S^2}|h_0-\bar{h}_0|+|\bar{h}_0|\norm{\eta_0+2F^2-F^{ij}\bar{R}_{3i3j}+\frac{1}{2}F^{ij}_{,ij}}_{L^2}\right)\\
    \leq & c\sigma^{-\alpha-2}|\bar{h}_0|,
\end{split}     
\end{equation}
for some $\alpha\in(0, 1)$.

Since $u_0=u-h_0$, \eqref{lu} and \eqref{fphi} imply
\begin{equation*}
    \begin{split}
        Su_0=& Su-Sh_0=-\frac{1}{2\sigma^2}\bar{u}+E_2+E_1-\eta_0h_0\\
        =& \frac{1}{2\sigma^2}(h_0-\bar{h}_0)-\frac{1}{2\sigma^2}\bar{u}_0+E_3+E_2+E_1,
    \end{split}
\end{equation*}
where $|E_3|\leq c\sigma^{-3}|\bar{h}_0|$.
Because $S$ has no kernel,
the Schauder estimate and \eqref{hal} gives
\begin{equation}\label{u0c2}
    \begin{split}
        \norm{u_0}_{C^{2, \alpha}}\leq & c(\sigma^{-2}\norm{h_0-\bar{h}_0}_{C^{0, \alpha}}+\sigma^{-2}|\bar{u}_0|+\sigma^{-3}\norm{u}_{C^{2, \alpha}}+\sigma^{-3}|\bar{h}_0|)\\
        \leq & c(\sigma^{-\alpha-4}|\bar{h}_0|+\sigma^{-2}|\bar{u}_0|+\sigma^{-3}(\norm{u_0}_{C^{2, \alpha}}+\norm{h_0}_{C^{2, \alpha}})+\sigma^{-3}|\bar{h}_0|).
    \end{split}
\end{equation}
Again since $h_0$ satisfies $Sh_0=\eta_0h_0$ and $\eta_0=O(\sigma^{-2})$, inequality \eqref{hal} leads to
\[ \norm{h_0}_{C^{2, \alpha}}\leq c\sigma^{-2}\norm{h_0}_{C^{0, \alpha}}\leq c\sigma^{-2}(\norm{h_0-\bar{h}_0}_{C^{0, \alpha}}+|\bar{h}_0|)\leq c\sigma^{-2}|\bar{h}_0|. \]
Inserting this into \eqref{u0c2}, we arrive at
\begin{equation}\label{u0}
    \norm{u_0}_{C^{2, \alpha}}\leq c(\sigma^{-3}|\bar{h}_0|+\sigma^{-2}|\bar{u}_0|),
\end{equation}
for $\sigma\geq\sigma_0$. Notice the fact that 
\[0=\int_{S^2}h_0u_0d\mu=\int_{S^2}u_0(h_0-\bar{h}_0)+\bar{h}_0\int_{S^2}u_0,\] 
this implies 
\[ |\bar{u}_0|\leq |\bar{h}_0|^{-1}\sup |h_0-\bar{h}_0|\sup |u_0|\leq c\sigma^{-2}|u_0|_{C^{2, \alpha}}. \]
Then combining with \eqref{u0} we arrive at the following inequality
\[ \norm{u_0}_{C^{2, \alpha}}\leq c\sigma^{-3}|\bar{h}_0|. \]
Therefore, we obtain
\[ |u-\bar{h}_0|\leq |u_0|+|h_0-\bar{h}_0|\leq c\sigma^{-2}|\bar{h}_0|,\]
which shows that $u$ has a sign since $\bar{h}_0$ has a sign.
Finally, an application of Lemma \ref{fol} completes the proof.

\end{proof}

We will now define the ``center of gravity" as the constant harmonic mean curvature foliation approaches infinity. Following the way of Corvino-Wu in \cite{Corvino}, we will show that the center of mass defined by the constant harmonic mean curvature foliation is equivalent to the ADM center of mass.
We first present some definitions and necessary propositions.

The second derivative of the second fundamental form can be similarly estimated as Proposition \ref{3.10}, see the appendix in \cite{Corvino}. Together with Lemma \ref{lem1.4}, we can derive the following lemma.
\begin{lem}\label{2of}
    Let $\Sigma_{\sigma}$ be the constant harmonic mean curvature surface constructed in Theorem \ref{main}. Then, for all $t\geq0$ and $\sigma\geq\sigma_0$ large enough,
    \[ |\n^2F|\leq c\sigma^{-5}, \]
    where $c$ depends on $C_0, B_1, B_2, B_3$.
\end{lem}

\begin{defn}
    Assume $(M,g)$ be an asymptotically Schwarzschild manifold with metric $g_{ij}=\left(1+\frac{m}{2|x|} \right)^4\delta_{ij}+O_2(|x|^{-2})$, the ADM center of mass is defined by 
    \[ C^k=\frac{1}{16m\pi}\lim_{R\to\infty}\int_{|x| = R}\left[\sum_{i}x^{k}\left(g_{ij,i}-g_{ii,j}\right)v_{e}^{j}d\mu_{e}-\sum_{i}\left(g_{ik}v_{e}^{i}-g_{ii}v_{e}^{k}\right)d\mu_{e}\right],\]
    where $\nu_e$ is the normal vector with respect to the Euclidean metric $\delta$, $k=1, 2, 3$.
\end{defn}

Consider an exterior region in an asymptotically flat manifold, and assume that there is an asymptotically flat chart in which $g_{ij}=\left(1+\frac{m}{2|x|}+\frac{B_{k}x^{k}}{|x|^{3}}\right)^{4}\delta_{ij}+O_{5}\left(|x|^{-3}\right)$ with $m>0$. By direct calculation, the ADM center of mass is 
\begin{equation}\label{center}
    C^k=\frac{2B_k}{m}. 
\end{equation}

\begin{defn}
    Let $\Sigma_{\sigma}$ be the family of surfaces constructed in Theorem \ref{local} and $\phi^{\sigma}$ be the position vector.
     The center of gravity of $\Sigma_{\sigma}$ is defined as
    \[ C_{HM}=\lim_{\sigma\rightarrow\infty}\frac{\int_{\Sigma_{\sigma}}\phi^{\sigma} d\mu_e}{\int_{\Sigma_{\sigma}}d\mu_e}. \]
\end{defn}

We note that translating the coordinates by a vector $a^k$ shifts both the ADM center of mass and the center $C_{HM}$ by $a^k$. Thus, we can adjust the coordinates to set the ADM center of mass vector to be zero, which results in $B_k=0$ in this chart by \eqref{center}. To prove Theorem \ref{thm-center} saying that $C_{HM}$ is equivalent to the ADM center of mass,  it suffices to consider the following case.  

\begin{thm}
    Consider an exterior region in an asymptotically flat manifold, and assume that there is an asymptotically flat chart in which $g_{ij}=\left(1+\frac{m}{2|x|}\right)^{4}\delta_{ij}+O_{5}\left(|x|^{-3}\right)$ with $m>0$.
    Then $C_{HM}=0$.
\end{thm}

\begin{proof}
    By the Sobolev embedding inequality ((3.1) in \cite{Corvino}), for $r>2$ and all $t\geq0$, there is a constant $c(r)$, such that 
\begin{equation}\label{inter}
    ||F-f||_{C^0}\leq c\sigma^{-\frac{2}{q}}(\sigma||\n F||_{L^q}+||F-f||_{L^r}^{\frac{1}{q}}||F-f||_{L^2}^{1-\frac{1}{q}}), 
\end{equation} 
where $q=\frac{r}{2}+1$ and norms are calculated on $\Sigma_t$ which is a solution to \eqref{flow} contained in $\mathcal{B}_{\sigma}(B_1, B_2, B_3)$.
At $t=0$, we are working in a coordinate system that $g_{ij}-g^S_{ij}=O_5(|x|^{-3})$, with $g^S_{ij}$ being the Schwarzschild metric. 
Thus, the difference between the second fundamental form of $\Sigma_t$ is 
\[ h_{ij}-h^S_{ij}=O(|x|^{-4}), \]
    which implies that $|F-f|=O(\sigma^{-4})$.
From \eqref{expo}, for all $t\geq0$,
\[ \int_{\Sigma_t}(F-f)^2d\mu_t\leq C\sigma^{-6}e^{-\frac{2mt}{\sigma^3}}. \]
To derive a pointwise bound from \eqref{inter}, we utilize the interpolation inequality from \cite{Hamilton} 
to estimate, for $\frac{1}{p}+\frac{1}{2}=\frac{2}{q}$,
\begin{equation*}
    \begin{split}
        \sigma^{1-\frac{2}{q}}||\n F||_{L^q}
       \leq & q\sigma^{1-\frac{2}{q}}||\n^2F||_{L^p}^{\frac{1}{2}}||F-f||_{L^2}^{\frac{1}{2}} \\
       \leq & C \sigma^{1-\frac{2}{q}}(\sigma^{-5}\sigma^{\frac{2}{p}})^{\frac{1}{2}}\cdot(\sigma^{-6}e^{-\frac{2mt}{\sigma^3}})^{\frac{1}{4}}\\
       = & C\sigma^{-\frac{7}{2}}e^{-\frac{mt}{2\sigma^3}},
    \end{split}
\end{equation*} 
and 
\begin{equation*}
    \begin{split}
        &\sigma^{-\frac{2}{q}}||F-f||_{L^r}^{\frac{1}{q}}||F-f||_{L^2}^{1-\frac{1}{q}}\\
        \leq & C\sigma^{-\frac{2}{q}}\cdot\left(\sigma^{-3}\sigma^{\frac{2}{r}}\right)^{\frac{1}{q}}\cdot\left(\sigma^{-6}e^{-2mt\sigma^{-3}}\right)^{\frac{1}{2}\left(1-\frac{1}{q}\right)}\\
        = & C\sigma^{-3-\frac{2}{q}\left(1-\frac{1}{r}\right)}\cdot\left(e^{-2mt\sigma^{-3}}\right)^{\frac{1}{2}\left(1-\frac{1}{q}\right)}\\
        \leq & C\sigma^{-\frac{7}{2}}e^{-\frac{mt}{2\sigma^3}}.
    \end{split}
\end{equation*}
In the last inequality, we have selected that $r\in (2, 4)$.
Therefore, combining with \eqref{inter} we get on $\Sigma_t$ that, for all $t\geq0$ and $\sigma\geq\sigma_0$,
\[ |F-f|\leq C\sigma^{-\frac{7}{2}}e^{-\frac{mt}{2\sigma^3}}. \]
By integrating the flow equation $\frac{\partial}{\partial t}\phi^{\sigma}=(f-F)\nu$, and $\partial_t(d\mu_g)=H(f - F)d\mu_g$, we obtain 
\begin{align*}
    |\phi^{\sigma}_{\infty}-\phi^{\sigma}_0|\leq & C\sigma^{-\frac{1}{2}},\\
    |(\phi^{\sigma}_{\infty})^*(d\mu_g)-(\phi^{\sigma}_0)^*(d\mu_g)|\leq & C\sigma^{-\frac{3}{2}}|(\phi^{\sigma}_0)^*(d\mu_g)|.
\end{align*}
  Finally, we have the center of mass 
  \[ C_{HM}=\lim_{\sigma\rightarrow\infty}\frac{\int_{\Sigma_{\sigma}}\phi^{\sigma} d\mu_e}{\int_{\Sigma_{\sigma}}d\mu_e}
  =\lim_{\sigma\rightarrow\infty}\frac{\int_{S^2}\phi^{\sigma}_{\infty}d\mu^{\sigma}}{\int_{S^2}d\mu^{\sigma}}=\lim_{\sigma\rightarrow\infty}O(\sigma^{-\frac{1}{2}})=0,\]
where $d\mu^{\sigma}=(\phi^{\sigma}_{\infty})^*(d\mu_e)$.
    
\end{proof}
\vspace{.2in}

\section{Appendix}
In this appendix, we give the details of the 1st variation of the $F$-operator. We begin with the calculation of the $H$-operator with respect to arbitrary variation. So let $V_t$ be the variation vector field, that is, 
\[
\pd t\phi_t=V_t.
\]
As before, we let $e_i(t)=d\phi_t(e_i)$, then $\pa_ te_i(t)=\bar{\n}_{e_i}V$. Recall the variation formula 
 \begin{enumerate}
         \item[i)]
$$
\begin{aligned}
\frac{d}{dt}g_{ij}=\<\bar{\n}_{e_i}V,e_j\>+\<e_i,\bar{\n}_{e_j}V\>,\\
 \f d{dt}g^{ij}=-g^{ik}g^{jl}(\<\bar{\n}_{e_k}V,e_l\>+\<e_k,\bar{\n}_{e_l}V\>)\\
\frac{d}{dt}\nu=-\<\nu,\bar{\n}_{e_i}V\>e_i.
\end{aligned}
$$
\item[ii)]
\begin{align*}
        \pd th_{ij} & = \pd t\<\bar{\n}_{e_i}\nu,e_j\>=\<\bar{\n}_{e_i}\pa_t\nu+\bar{R}(V_t,e_i)\nu,e_j\>+\<\bar{\n}_{e_i}\nu,\pa_te_j\>\\
        & = -\<\nu,\bar{\n}_{e_k}V\>\<\bar{\n}_{e_i}e_k,e_j\>-\<\nu,\bar{\n}_{e_i}\bar{\n}_{e_j}V\>+\bar{R}(V,e_i,\nu,e_j)\\
        & = -\<\nu,\bar{\n}^2V(e_i,e_j)\>+\bar{R}(V,e_i,\nu,e_j).
 \end{align*}
  \end{enumerate}
Now, the mean curvature operator is given by 
\begin{align*}
    \pd tH=\pd t(g^{ij}h_{ij}) & = \pd tg^{ij}h_{ij}+g^{ij}\pd th_{ij}\\
    & = -h^{kl}\left(\<\bar{\n}_{e_k}V,e_l\>+\<e_k,\bar{\n}_{e_l}V\>\right)-\bar{R}ic(V,\nu)-\<\nu,\Delta V\>
\end{align*}
Note that 
\begin{align*}
    \Delta\<\nu,V\> & = \<\nu,\bar{\n}_{e_i}\bar{\n}_{e_i}V\>+\<V,\bar{\n}_{e_i}\bar{\n}_{e_i}\nu\>+2\<\bar{\n}_{e_i}V,\bar{\n}_{e_i}\nu\>\\
    & = \<\nu,\bar{\n}_{e_i}\bar{\n}_{e_i}V\>+\<V,\bar{\n}_{e_i}(h_{ij}e_j)\>+2\<\bar{\n}_{e_i}V,\bar{\n}_{e_i}\nu\>\\
    & = \<\nu,\bar{\n}_{e_i}\bar{\n}_{e_i}V\>+2\<\bar{\n}_{e_i}V,\bar{\n}_{e_i}\nu\>+h_{ij,i}\<e_j,V\>+h_{ij}\<V,\bar{\n}_{e_i}e_j\>\\
    & = \<\nu,\bar{\n}_{e_i}\bar{\n}_{e_i}V\>+2\<\bar{\n}_{e_i}V,\bar{\n}_{e_i}\nu\>+\<H_je_j,V\>-\bar{R}_{3iji}\<e_j,V\>-\abs{A}^2\<V,\nu\>,
\end{align*}
and 
\[
\bar{R}ic(V,\nu)=\bar{R}ic(\nu,\nu)\<V,\nu\>+\bar{R}ic(e_j,V)\<e_j,V\>.
\]
We thus get 
\[
\pd tH=-\Delta\<V,\nu\>-\left(\bar{R}ic(\nu,\nu)+\abs{A}^2\right)\<V,\nu\>+\<\n H,V\>.
\]
Now we proceed the computation of the $F$-operator. Note first that 
\begin{align*}
    F^{ij}\<V,\nu\>_{ij} & =F^{ij}\left(\<\nu,\bar{\n}_{e_i}\bar{\n}_{e_j}V\>+\<V,\bar{\n}_{e_i}\bar{\n}_{e_j}\nu\>+2\<\bar{\n}_{e_i}V,\bar{\n}_{e_j}\nu\>\right)\\
    & = F^{ij}\<\nu,V_{ij}\>+F^{ij}h_{jk,i}\<V,e_k\>-F^{ij}h_{jk}h_{ik}\<V,\nu\>+2F^{ij}\<\bar{\n}_{e_i}V,\bar{\n}_{e_j}\nu\>\\
    & =  F^{ij}\<\nu,V_{ij}\>+\left(\<V,\n F\>-F^{ij,pq}h_{pq,k}h_{ij}\<V,e_k\>-F^{ij}\bar{R}_{3jki}\<V,e_k\>\right)\\
    & -F^{ij}h_{jk}h_{ik}\<V,\nu\>+2F^{ij}\<\bar{\n}_{e_i}V,\bar{\n}_{e_j}\nu\>
\end{align*}
We hence obtain 
\begin{align*}
    \pd tF  =& F^{ij}\pd t h_{ij}+F^{ij}\pd tg^{ik}h_{kj}\\
     =& -F^{ij}\<\nu,V_{ij}\>+F^{ij}\bar{R}(V,e_i,\nu,e_j)-F^{ij}h_{kj}(g^{ip}g^{kl}(\<\bar{\n}_{e_p}V,e_l\>+\<e_p,\bar{\n}_{e_l}V\>))\\
     =& -\mathcal{L}\<V,\nu\>-F^{ij}\left(h_{jk}h_{ik}-\bar{R}(V,e_i,\nu,e_j)\right)\<V,\nu\>\\
    &+\<V,\n F\>-F^{ij,pq}h_{pq,k}h_{ij}\<V,e_k\>.
   % & +2F^{ij}\<\bar{\n}_{e_i}V,\bar{\n}_{e_j}\nu\>-%F^{ij}h_{kj}(\<\bar{\n}_{e_i}V,e_k
   % \>+\<e_i,\bar{\n}_{e_k}V\>)).
\end{align*}
We now conclude that the 1st variation of the $F$-operator with the constant harmonic mean curvature is given by 
\[
d\mathscr{F}(\phi)\cdot V=-\left(\mathcal{L}+F^{ij}h_{jk}h_{ik}-F^{ij}\bar{R}(\nu,e_i,\nu,e_j)\right)\<V,\nu\>,
\]
where $\phi:S^2\rightarrow{N}$ is the embedding of the constant harmonic mean curvature surface.
Define the operator 
\[
L=-\left(\mathcal{L}+F^{ij}h_{jk}h_{ik}-F^{ij}\bar{R}(\nu,e_i,\nu,e_j)\right).
\]
This operator is then interpreted as the first variation of the constant harmonic mean curvature surface in the normal direction.

\vspace{.2in}
\textbf{Conflict of Interest} The authors have no conflict of interest to declare.

\vspace{.1in}
\textbf{Data availability} The authors declare no datasets were generated or analysed during the current study.

\begin{comment}
    \begin{thm}
    Any energy minimising harmonic map $u:\mathbb{H}^n\rightarrow N$, from the Heisenberg group to a non-positively curved space N CAT(0) is locally Lipschitz continuous. Moreover, for any $x\in\mathbb{H}^n$, we can find some constant $C=C(n,N)$ such that there holds the estimate
\end{thm}
\[
\text{Lip}_u(x)\leq C(n,N)E^u.
\]
\end{comment}

\bibliography{refofmix}
\bibliographystyle{alpha}

\end{document}